\documentclass{amsart}

\usepackage{graphicx}
\usepackage{overpic}
\usepackage{enumitem}
\usepackage{braket}
\usepackage{amsmath, amssymb, amsthm, amsfonts, latexsym, mathtools}

\usepackage[top=3.5cm, bottom=2.5cm, left=3cm, right=3cm, includefoot]{geometry}

\newtheorem{theorem}{Theorem}[section]
\newtheorem{corollary}[theorem]{Corollary}
\newtheorem{proposition}[theorem]{Proposition}
\newtheorem{lemma}[theorem]{Lemma}
\newtheorem{claim}[theorem]{Claim}

\theoremstyle{remark}

\theoremstyle{definition}

\renewcommand{\(}{\left(}
\renewcommand{\)}{\right)}

\DeclareMathOperator{\lk}{lk}

\begin{document}

\title{Large alternating Montesinos knots do not admit purely cosmetic surgeries}

 \author{Kazuhiro Ichihara}
 \address{Department of Mathematics, College of Humanities and Sciences, Nihon University, 3-25-40 Sakurajosui, Setagaya-ku, Tokyo 156-8550, JAPAN}
 \email{ichihara.kazuhiro@nihon-u.ac.jp}

 \author{In Dae Jong}
 \address{Department of Mathematics, Kindai University, 3-4-1 Kowakae, Higashiosaka City, Osaka 577-0818, Japan} 
 \email{jong@math.kindai.ac.jp}

\dedicatory{Dedicated to Professor Masakazu Teragaito on his 60th birthday.}

\date{\today}



% \thanks{Ichihara is partially supported by JSPS KAKENHI Grant Number 18K03287. Jong is partially supported by JSPS KAKENHI Grant Number 19K03483.}

\begin{abstract}
It is conjectured that, on a non-trivial knot in the $3$--sphere, no pair of Dehn surgeries along distinct slopes are purely cosmetic, that is, none of them yield $3$--manifolds those are orientation-preservingly homeomorphic. 
In this paper, we show that alternating knots having reduced alternating diagram with the twist number at least 7, Montesinos knots of length at least 5, and alternating Montesinos knots of length at least 4 do not admit purely cosmetic surgeries. 
As a corollary, we see that large alternating Montesinos knots have no purely cosmetic surgeries. 
\end{abstract}

\maketitle

\section{Introduction}

For a knot $K$ in the 3--sphere $S^3$, the following operation is called a \textit{Dehn surgery}: 
take the exterior $E(K)$ of $K$ and glue a solid torus onto the torus $\partial E(K)$. 
The \textit{surgery slope} of the Dehn surgery on $\partial E(K)$ is represented by the curve identified with the meridian of the attached solid torus.

When we consider a Dehn surgery as an operation to create a variety of 3--manifolds from a given knot, the next should be one of the fundamental conjectures in the study of Dehn surgery. 
The \emph{Purely Cosmetic Surgery Conjecture} (PCSC for short) \cite{StipsiczSzabo}: 
Dehn surgeries along two distinct slopes on a non-trivial knot in $S^3$ do not yield $3$--manifolds those are orientation-preservingly homeomorphic. 
Such Dehn surgeries are called \emph{purely cosmetic}, and have been studied extensively. 
(See \cite[Problem 1.81(A)]{Kirby} for its generalization to general 3-manifolds and further information.) 
In fact, for several infinite families of knots, PCSC has been confirmed; 
for knots of genus one~\cite{Wang}, cable knots~\cite{TaoCable}, two-bridge knots~\cite{IchiharaJongMattmanSaito}, alternating fibered knots~\cite{IchiharaJongMattmanSaito}, pretzel knots~\cite{StipsiczSzabo}, 3-braid knots~\cite{Varvarezos21}, and composite knots~\cite{TaoComposite}. 

In this paper, %we give computations of invariants for knots not to have purely cosmetic surgeries. In particular, 
we focus on alternating knots and Montesinos knots as an enlargement of the previous researches~\cite{IchiharaJongMattmanSaito,StipsiczSzabo}. 
A knot in $S^3$ is called \emph{alternating} if it admits a diagram with alternatively arranged over- and under-crossings running along it.
Our first result is the following. 

\begin{theorem}\label{thm:alt}
If a non-trivial alternating knot $K$ admits purely cosmetic surgeries, then $K$ is equivalent to one of the knots depicted in Figure~\ref{fig:alt}.
Here $a, b, c, d, e, f$ are non-negative integers.
\end{theorem}

\begin{figure}
	\centering
	\begin{overpic}[width=.8\textwidth]{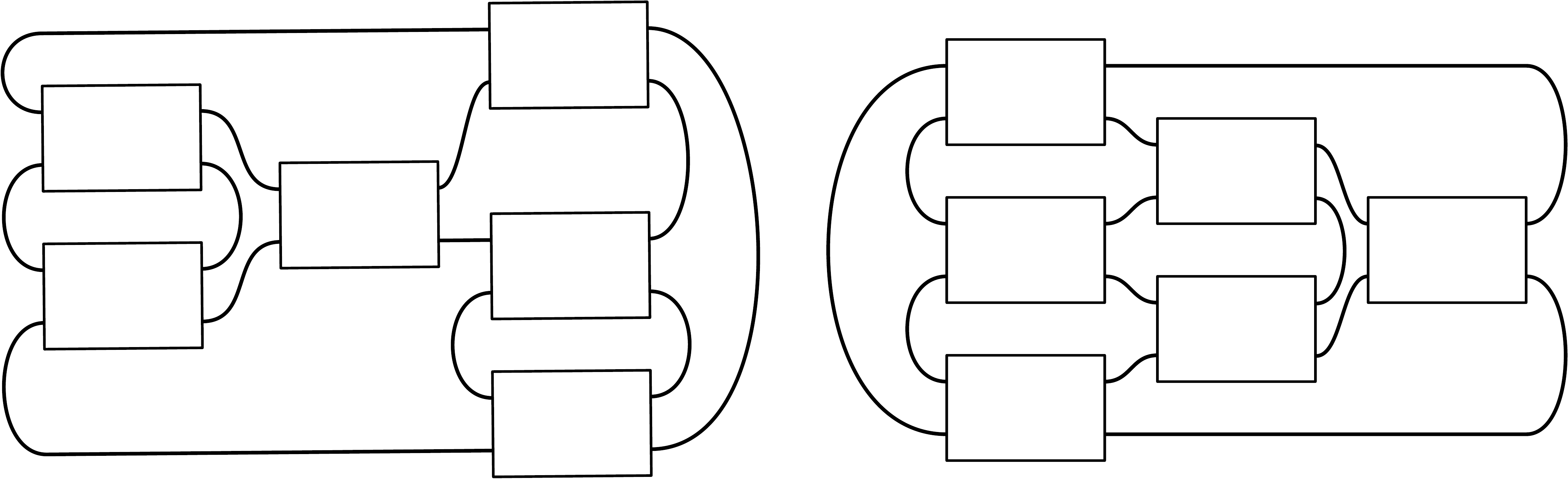}
		\put(4,20.8){$2c+1$}
		\put(4,10.5){$2b+1$}
		\put(19.1,16){$2a+1$}
		\put(31.3,26){$-2d-1$}
		\put(31.4,12.5){$-2e-1$}
		\put(31.4,2.5){$-2f-1$}
		\put(60.4,23.8){$-2a-2$}
		\put(60.4,13.5){$-2b-2$}
		\put(60.4,3.5){$-2c-2$}
		\put(75.1,18.7){$2d+1$}
		\put(75.1,8.7){$2e+1$}
		\put(88.5,13.7){$2f+1$}
	\end{overpic}
	\caption{Alternating knots in Theorem~\ref{thm:alt}.}
	\label{fig:alt}
\end{figure}
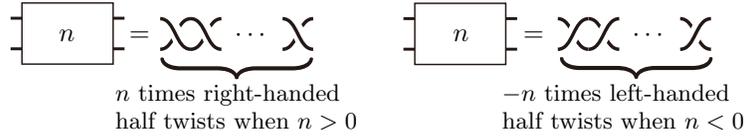
\begin{figure}[!htb]
	\centering
	\begin{overpic}[width=.6\textwidth]{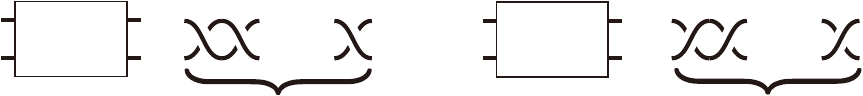}
		\put(7,5.5){$n$} 
		\put(17,5.5){$=$} 
		\put(32,5.5){$\cdots$} 
		\put(63,5.5){$n$} 
		\put(73,5.5){$=$} 
		\put(88.5,5.5){$\cdots$} 
		\put(15,-3){{\small $n$ times right-handed}}
		\put(15,-7){{\small half twists when $n>0$}}
		\put(70,-3){{\small $-n$ times left-handed}}
		\put(70,-7){{\small half twists when $n<0$}}
	\end{overpic}
	\bigskip\bigskip
	\caption{The right-handed (resp.\ left-handed) twists when $n >0$ (resp.\ $n<0$).} 
	\label{fig:twists}
\end{figure}

%Each boxes in Figure~\ref{fig:alt} is called a \emph{twist box} which indicates two-strand twists as in Figure~\ref{fig:twists}. 
We call each box in Figure~\ref{fig:alt} a \emph{twist box} which indicates two-strand twists as shown in Figure~\ref{fig:twists}.
%By PCSC, It is expected that all of the knots in Figure~\ref{fig:alt} admit no purely cosmetic surgeries. 
In general, the \emph{twist number} of a link diagram $D$ %, denoted by $t(D)$, 
is defined as the number of \emph{twists}, which are either maximal connected collections of bigon regions in $D$ arranged in a row or isolated crossings adjacent to no bigon regions.
Then we have the following corollary. 

\begin{corollary}\label{cor:alt}
	Alternating knots having reduced alternating diagram with the twist number at least 7 do not admit purely cosmetic surgeries. 
\end{corollary}

Next we consider Montesinos knots. 
A \emph{Montesinos knot} is a knot obtained by closing a tangle, called a Montesinos tangle, which is obtained by putting rational tangles together in a horizontal way. 
The next is our second result. 

\begin{theorem}\label{thm:Montesinos}
If a non-trivial Montesinos knot $K$ admits purely cosmetic surgeries, then $K$ is equivalent to one of the following.%the knots depicted in Figure~\ref{}.
\begin{itemize}
\item[(o1)] 
% $M\left([2a+1, 2b], [2c+1], [2d+1], [2e+1] \right)$ \quad $(a,b,c,d \in \mathbb{Z} \setminus \{0\} \text{ and } a,c,d \ne -1)$
% $(a,c,d \ne 0,-1 \text{ and } b\ne 0)$
$M\left(\dfrac{2b}{4ab+2b-1}, \dfrac{1}{2c+1}, \dfrac{1}{2d+1}, \dfrac{1}{2e+1} \right)$ \quad $(a,b,c,d \in \mathbb{Z} \setminus \{0\} \text{ and } a,c,d \ne -1)$
\item[(o2)] 
% $M\left([2a+1, 2b], [2c+1, 2d], [2e+1] \right)$  \quad $(a,b,c,d,e \in \mathbb{Z} \setminus \{0\}  \text{ and }  a,c,e \ne -1)$
% $(a,c,e \ne 0, -1 \text{ and } b, d \ne 0 )$
$M\left(\dfrac{2b}{4ab+2b-1}, \dfrac{2d}{4cd+2d-1}, \dfrac{1}{2e+1} \right)$  \quad $(a,b,c,d,e \in \mathbb{Z} \setminus \{0\}  \text{ and }  a,c,e \ne -1)$
\item[(o3)] 
% $M\left([2a, \pm3], [2b+1] ,[2c+1] \right)$ \quad $(a,b,c \in \mathbb{Z} \setminus \{0\}  \text{ and }  b,c \ne -1)$
$M\left(\dfrac{3}{6a \mp 1}, \dfrac{1}{2b+1} , \dfrac{1}{2c+1} \right)$ \quad $(a,b,c \in \mathbb{Z} \setminus \{0\}  \text{ and }  b,c \ne -1)$
\item[(o4)] 
% $M\left([2a, \pm2, 2b+1], [2c+1] ,[2d+1] \right)$  \quad $(a,b,c,d \in \mathbb{Z} \setminus \{0\} \text{ and }  b,c,d \ne -1)$
% $(b,c,d \ne 0,-1 \text{ and } a \ne 0)$
$M\left(\dfrac{4b +2 \mp 1}{8ab +4a \mp 2a \mp 2b \mp 1}, \dfrac{1}{2c+1}, \dfrac{1}{2d+1} \right)$  \quad $(a,b,c,d \in \mathbb{Z} \setminus \{0\} \text{ and }  b,c,d \ne -1)$
\item[(o5)] 
% $M\left([2a+1, 2b, 2c], [2d+1] ,[2e+1] \right)$  \quad $(a,b,c,d,e \in \mathbb{Z} \setminus \{0\} \text{ and }  a,d,e \ne -1)$
% $(a,d,e \ne 0,-1 \text{ and } b, c \ne 0)$
$M\left(\dfrac{4bc-1}{8abc + 4bc -2a -2c -1}, \dfrac{1}{2d+1} ,\dfrac{1}{2e+1} \right)$  \quad $(a,b,c,d,e \in \mathbb{Z} \setminus \{0\} \text{ and }  a,d,e \ne -1)$
\item[(e1)] 
% $M\left([2a], [2b, 2c] ,[2d, 2e] \right)$ \quad $(a,b,c,d,e \in \mathbb{Z} \setminus \{0\})$
$M\left(\dfrac{1}{2a}, \dfrac{2c}{4bc-1}, \dfrac{2e}{4de-1} \right)$ \quad $(a,b,c,d,e \in \mathbb{Z} \setminus \{0\})$
\end{itemize} 
\end{theorem}

The notation used in the statement above will be explained in Section~\ref{sec:Montg2}.  
The minimal number of rational tangles forming a Montesinos knot is called the \emph{length} of the Montesinos knot. 
The following is an immediate corollary of Theorem~\ref{thm:Montesinos}. 

\begin{corollary}\label{cor:Montesinos}
	Montesinos knots of length at least 5 do not admit purely cosmetic surgeries. 
\end{corollary}

Furthermore, based on Theorems~\ref{thm:alt} and \ref{thm:Montesinos}, we obtain the following. 

\begin{theorem}\label{thm:altMontesinos}
	% If a non-trivial alternating Montesinos knot $K$ admits purely cosmetic surgeries, then $K$ is equivalent to a knot depicted in Figure~\ref{fig:altMontesinos}. 
	If a non-trivial alternating Montesinos knot $K$ admits purely cosmetic surgeries, then $K$ is equivalent to the knot in the left side of Figure~\ref{fig:alt}, where $a, b, c, d, e, f$ are non-negative integers such that either 
    \begin{itemize}
    \item $c=d=0$ and $e,f \ge 1$, or 
    \item $b,c,e \ge 1$ and $d=f=0$, or 
    \item $a,b,c,d,f \ge 1$ and $e=0$. 
    \end{itemize}
\end{theorem}

As a corollary of Theorem~\ref{thm:altMontesinos}, we obtain the following. 

\begin{corollary}\label{cor:altMontesinos}
    Alternating Montesinos knots of length at least 4 admit no purely cosmetic surgeries. 
\end{corollary}

% \begin{figure}[htb]
%     \centering
% 	\begin{overpic}[width=.4\textwidth]{fig_6_3.pdf}
% 		\put(8.5,36.5){$2c+1$}
% 		\put(8.5,15.5){$2b+1$}
% 		\put(39.5,26){$2a+1$}
% 		\put(65.3,5){$-2d-1$}
% 		\put(65.3,54){$-2f-1$}
% 		\put(65.3,33){$-2e-1$}
% 	\end{overpic}
% 	\caption{Alternating Montesinos knots in Theorem~\ref{thm:altMontesinos}.}
% 	\label{fig:altMontesinos}
% \end{figure}

A knot is called \textit{small} if its exterior contains no closed essential surfaces, and is said to be \textit{large} if it is not small. 
%A Montesinos knot is called \emph{large} if it has a Conway sphere, equivalently, its complement is large in the sense that it contains an embedded closed essential surface. 
It is known that if a Montesinos knot is large, then it is of length at least 4. 
See \cite[Corollary 4]{Oertel}. 
Thus the corollary above implies that large alternating Montesinos knots have no purely cosmetic surgeries. 

% Our argument depends on a result by Hanselman~\cite{Hanselman} together with several invariants of knots. 

%\bigskip

%Here we recall basic definitions and terminologies about Dehn surgery.  
% For a knot $K$ in the 3-sphere $S^3$, the following operation is called a \textit{Dehn surgery}: 
% take the exterior $E(K)$ of $K$ and glue a solid torus onto the peripheral torus $\partial E(K)$. 
% The \textit{surgery slope} of the Dehn surgery on $\partial E(K)$ is represented by the curve identified with the meridian of the attached solid torus.
%Using the standard meridian-longitude system, slopes on the peripheral torus are parametrized by rational numbers along with $1/0$, which corresponds to the meridian. 
%When a slope $\gamma$ corresponds to a rational number $r$, Dehn surgery along $\gamma$ is called \textit{$r$-Dehn surgery}, or simply \textit{$r$-surgery}. 
% The resultant manifold is denoted by $K(r)$. 

\subsection*{Organization}
In Section~\ref{sec:obst}, we introduce some knot invariants which can be used to show that a knot admits no purely cosmetic surgeries. 
In Section~\ref{sec:Montg2}, we give a list of non-pretzel and non-two-bridge Montesinos knots of genus two. 
In Section~\ref{sec:proof}, we prove Theorems~\ref{thm:alt}, \ref{thm:Montesinos}, \ref{thm:altMontesinos}.

\section{Obstructions}\label{sec:obst} 

In this section, we introduce knot invariants used to prove Theorems~\ref{thm:alt}, \ref{thm:Montesinos}, and \ref{thm:altMontesinos}. 
These invariants are used to show that a knot admits no purely cosmetic surgeries. 
Hereafter, we assume that a knot $K$ is non-trivial and prime unless otherwise noted.

\subsection{Genus obstruction}
For a knot $K$ in $S^3$, $g(K)$ denotes the genus of $K$, that is, the minimal genus of Seifert surfaces for $K$. 

\begin{lemma}\label{lem:Hanselman}
	Let $K$ be an alternating knot or a Montesinos knot. 
	If $g(K) \ne 2$, then PCSC is true for $K$. 
\end{lemma}
\begin{proof}
	Suppose that $g(K) \ne 2$. 
	Since PCSC is true for knots of genus one~\cite{Wang}, we may assume that $g(K) > 2$. 
	Hanselman~\cite{Hanselman} proved that PCSC is true for any knot $J$ in $S^3$ with $g(J) >2$ and $th(J) < 6$. 
	Here $th(J)$ is the Heegaard Floer thickness of $J$. 
	Thus, it suffices to show that $th(K) <6$ when $K$ is alternating or Montesinos. 
	For any knot $J$ in $S^3$, $th(J)\le g_T(J)$ holds~\cite{Lowrance}\footnote{$th(J) -1 = w_{HF}(J)$ holds. For the definition of $w_{HF}$, see \cite{Lowrance}.}, where $g_T(J)$ is the Turaev genus of $J$. 
	If $K$ is alternating, then $g_T(K) = 0$~\cite[Corollary 4.6]{DasEtAl}, and thus, $th(K) =0$ holds. 
	If $K$ is a non-alternating Montesinos knot, 
	%then $K$ is almost alternating, and 
	then $g_T(K) = 1$~\cite[Corollary 5.4]{AbeJongKishimoto}, and thus, $th(K) \le 1$ holds. 
\end{proof}

% By the same argument, we can prove the following. 

% \begin{lemma}\label{lem:HanselmanA}
% 	PCSC is true for an alternating knot $K$ with $g(K) \ne 2$. \qed
% \end{lemma}

% By Lemmas~\ref{lem:HanselmanM}, \ref{lem:HanselmanA}, we may assume that the genus is two to study PCSC for alternating knots or Montesinos knots. 

\subsection{Finite type invariant obstruction}

We denote by $V_K(t)$ the Jones polynomial of a knot $K$. 
The first author and Wu proved the following as an improvement of a result by Boyer and Lines~\cite[Proposition 5.1]{BoyerLines}. 

\begin{lemma}[{\cite[Theorem 1.1]{IchiharaWu}}]\label{lem:IchiharaWu}
	If a knot $K$ satisfies either $V_K''(1) \ne 0$ or $V_K'''(1) \ne 0$, then PCSC is true for $K$. 
\end{lemma}

It is known that for any knot $K$, 
\[V_K''(1) = -6 a_2(K) \] 
holds, where $a_2(K)$ is the second coefficient of the Conway polynomial of $K$. 
Suppose $(K_+, K_-, L_0)$ is a skein triple of knots as depicted in Figure~\ref{fig:SkeinTriple}. 
For any oriented knot, the following skein relation holds: 
\begin{equation}\label{eq:a2skein}
a_2(K_+) -a_2(K_-) = \mathrm{lk} (L_0)\, , 
\end{equation}
where $\mathrm{lk}(L)$ denotes the linking number of a two-component link $L$. 

\begin{figure}[htb]
    \centering
	\begin{overpic}[width=.4\textwidth]{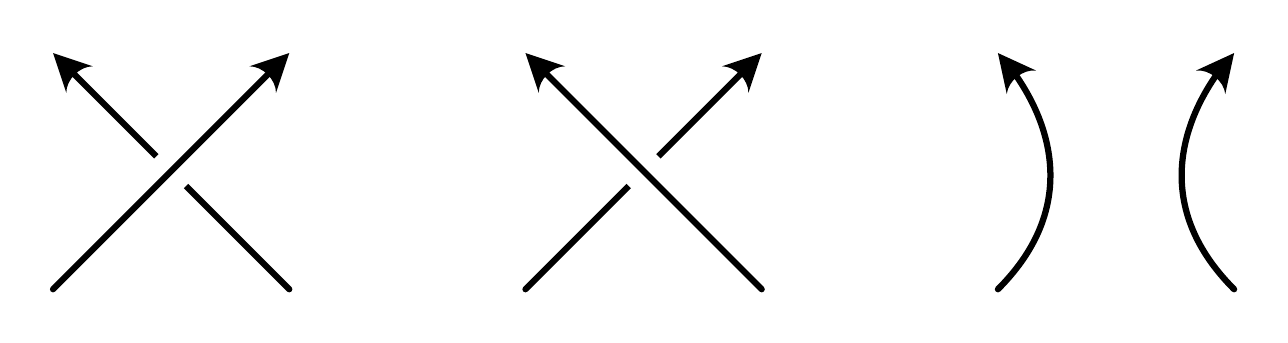}
		\put(9,-3){$K_+$}
		\put(45,-3){$K_-$}
		\put(82,-3){$L_0$}
	\end{overpic}
	\caption{The diagrams of $K_+$, $K_-$, $L_0$ are identical except at one crossing.}
	\label{fig:SkeinTriple}
\end{figure}

In \cite{Lescop}, Lescop defined an invariant $w_3$ for a knot $K$ in a homology sphere $Y$. 
When $Y = S^3$, the knot invariant $w_3$ satisfies the following crossing change formula. 
\begin{equation}\label{eq:w3skein}
	w_3(K_+) - w_3(K_-) = \dfrac{a_2(K') + a_2(K'')}{2} - \dfrac{a_2(K_+) + a_2(K_-) + \mathrm{lk}^2(K',K'')}{4}, 
\end{equation}
where $(K_+, K_-, K \cup K'')$ is the skein triple consisting of two knots $K_{\pm}$ and a two-component link $L_0 = K' \cup K''$ ~\cite[Proposition 7.2]{Lescop}. 

It is shown in \cite[Lemma 2.2]{IchiharaWu} that for any knot $K$ in $S^3$, 
\[ w_3(K) = \dfrac{1}{72}V_K'''(1) + \dfrac{1}{24}V_K''(1) \, . \]
Thus, Lemma~\ref{lem:IchiharaWu} is equivalent to the following. 

\begin{lemma}\label{lem:IchiharaWu2}
	If a knot $K$ satisfies either $a_2(K) \ne 0$ or $w_3(K) \ne 0$, then PCSC is true for $K$. 
\end{lemma}

Let $\mathrm{DT}(2x, 2y)$ be the \emph{double twist knot of type} $(2x, 2y)$ as shown in Figure~\ref{fig:tw}, where $x,y$ are integers. 
% Since double twist knots will appear frequently in the calculations in Section~\ref{sec:proof}, we here give formulae for some knot invariants for them. 
% The first one is of $a_2$ for them. 
% We here give formulae for some knot invariants, which is used in Section~\ref{sec:proof}. 
% Since double twist knots will appear frequently in the calculations in Section~\ref{sec:proof}, 
We here give formulae used in Section~\ref{sec:proof}. 

\begin{lemma}\label{lem:a2TK}
	$a_2\(\mathrm{DT}(2x,2y) \) = xy$. 
\end{lemma}
\begin{lemma}\label{lem:w3TK}
	$w_3(\mathrm{DT}(2x,2y))=\dfrac{xy(x+y)}{4}$. 
\end{lemma}
We omit proofs of Lemmas~\ref{lem:a2TK}, \ref{lem:w3TK} since one can prove them by elementary calculations. 
See \cite[Proposition 4.1, 4.4]{IchiharaWu} for example.

\begin{figure}[htb]
    \centering
	\begin{overpic}[width=.15\textwidth]{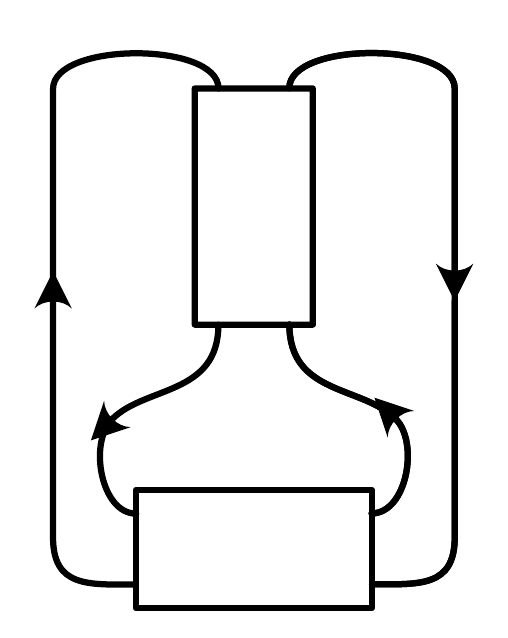}
		\put(34,65){$2x$}
		\put(34,10.5){$2y$}
	\end{overpic}
	\caption{$\mathrm{DT}(2x, 2y)$ called a double twist knot}
	\label{fig:tw}
\end{figure}

% We use the following formula to calculate the invariant $w_3$ of knots containing certain full twists. 

\begin{lemma}[A twist formula for $w_3$]\label{lem:w3twists}
	Let $K_n$ be a knot admits a diagram which contains anti-parallel $n$-full twists as in Figure~\ref{fig:AntiParallel}. 
	Suppose that each component of the two-component link $K' \cup K''$ obtained by smoothing a crossing in the $n$-full twists is a trivial knot. 
	Let $\lk = \lk(K',K'')$. 
	Then 
	\begin{equation}\label{eq:TwistFormula}
		w_3(K_n) = w_3(K_0) + \dfrac{n}{2} a_2(K_0) + \dfrac{n}{4} \lk\(\lk -n\) \, . 
	\end{equation}
\end{lemma}
\begin{proof}
	By the assumption that $K'$ and $K''$ are trivial, $a_2(K') = a_2(K'') = 0$. 
	Then, by Equation~\eqref{eq:w3skein}, 
	\begin{equation}\label{eq:w3skeinU}
		w_3(K_+) - w_3(K_-) = -\dfrac14\(a_2(K_+) + a_2(K_-) + \mathrm{lk}^2(K',K'')\). 
	\end{equation}
	Assume that $n \ge 0$. 
	Then the crossings in the $n$-full twists are negative. 
	Applying Equation~\eqref{eq:w3skeinU} repeatedly, we have 
	\begin{align}
		w_3(K_n) 
		&= w_3(K_{n-1}) + \dfrac{1}{4}\( a_2(K_n) + a_2(K_{n-1}) + \lk^2 \) \notag \\ 
		&= w_3(K_{n-2}) + \dfrac{1}{4}\( a_2(K_{n-1}) + a_2(K_{n-2}) + \lk^2 \) + \dfrac{1}{4}\( a_2(K_n) + a_2(K_{n-1}) + \lk^2 \) \notag \\ 
		&= \cdots \notag \\
		&= w_3(K_0) + \dfrac14 \sum_{k=0}^{n-1} \( a_2(K_{k+1}) + a_2(K_{k}) + \mathrm{lk}^2 \) \notag \\ 
		&= w_3(K_0) + \dfrac12 \sum_{k=1}^{n-1} a_2(K_k) + \dfrac14\( a_2(K_{0}) + a_2(K_{n}) \) + \dfrac{n}{4}\lk^2 .  \label{eq:lemTF}
	\end{align}
	By Equation~\eqref{eq:a2skein}, $a_2(K_k) = a_2(K_0) - k\lk$. Then RHS of Equation \eqref{eq:lemTF} equals 
	\begin{align*}
		&w_3(K_0) + \dfrac12 \sum_{k=1}^{n-1} \(a_2(K_0) - k\lk\) + \dfrac14\( a_2(K_{0}) + a_2(K_0) - n\lk \) + \dfrac{n}{4}\lk^2 \\ 
		&= w_3(K_0) + \dfrac{n-1}{2} a_2(K_0) - \dfrac{(n-1)n}{4}\lk + \dfrac14 a_2(K_{0}) + \dfrac14 a_2(K_{0}) - \dfrac{n}{4}\lk  + \dfrac{n}{4}\lk^2 \\ 
		&= w_3(K_0) + \dfrac{n}{2} a_2(K_0) + \dfrac{n}{4}\lk\(\lk -n\) . 
	\end{align*}
	Now we complete the proof in the case where $n \ge 0$. 
	We omit the proof in the case where $n \le 0$ since it is achieved by a similar way. 
\end{proof}

\begin{figure}[htb]
    \centering
	\begin{overpic}[width=.45\textwidth]{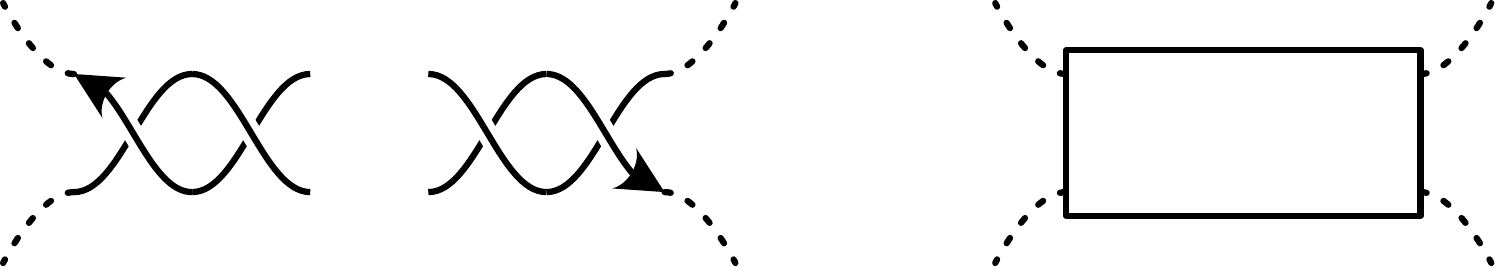}
		\put(21.5,7){$\cdots$}
		\put(56,7){$=$}
		\put(80,7){$2n$}
	\end{overpic}
	\caption{Anti-parallel $n$-full twists}
	\label{fig:AntiParallel}
\end{figure}

\subsection{$\tau$ invariant obstruction}

Another known class of knots having no purely cosmetic surgeries are that of knots with $\tau (K) \ne 0$~\cite[Theorem 1.2 (c)]{NiWu}, 
where $\tau$ is the concordance invariant defined by Ozs\'vath-Szab\'o and Rasmussen using Floer homology. 
It is known that if $K$ is alternating, then 
\begin{equation}\label{eq:tau-sig}
	\tau(K) = -\sigma(K)/2\, ,  
\end{equation}
where $\sigma(K)$ is the signature of $K$. 
Thus, for alternating knots, the signature is an obstruction to admit purely cosmetic surgeries. 
The signature of an alternating knot can be calculated diagramatically due to the following lemma. 

\begin{lemma}[{\cite[Proposition 3.11]{Lee}, \cite[Theorem 2(1)]{Traczyk}}]\label{lem:signature}
	For an oriented non-split alternating link $L$ and a reduced alternating diagram $D$ of $L$, we have 
	\begin{equation}\label{eq:Sig}
		\sigma(L) = o(D) - y(D) - 1\, . 
	\end{equation}
	Here $o(D)$ is the number of components of the diagram obtained by 0-resolutions of pattern $A$ (see Figure~\ref{fig:0-resolution}) at all the crossings of $D$, and $y(D)$ is the number of positive crossings of $D$. 
\end{lemma}

\begin{figure}[htb]
    \centering
	\begin{overpic}[width=.22\textwidth]{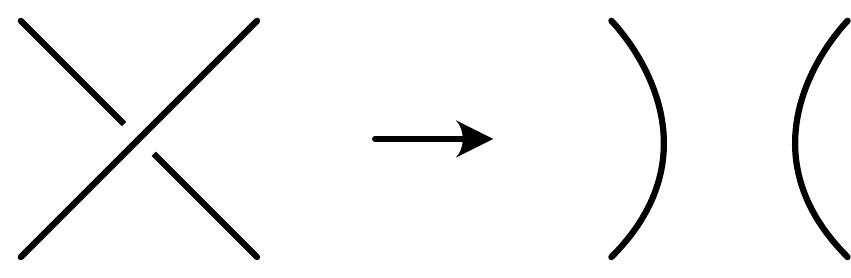}
	\end{overpic}
	\caption{$0$-resolution of pattern $A$}
	\label{fig:0-resolution}
\end{figure}

Another method to evaluate the signature of a knot is using a banding. 
For a given link $L$ in $S^3$ and an embedding $b \colon I \times I \to S^3$ such that $b ( I \times I ) \cap L = b ( I \times \partial I )$, where $I$ denotes a closed interval, we obtain a (new) link as $ ( L - b ( I \times \partial I ) ) \cup b ( \partial I \times I )$. 
We call the link so obtained the link obtained from $L$ by a \emph{banding}\footnote{The operation is sometimes called a band surgery, a bund sum (operation), or a hyperbolic transformation in a variety of contexts.} along the band $b$. 

\begin{lemma}[{\cite[Lemma 7.1]{Murasugi65}}]\label{lem:signatureBand}
	Let $K$ be a knot and $L$ a two-component link $L$ obtained from $K$ by a banding. 
    Then $|\sigma(K) - \sigma(L)| \le 1$. 
\end{lemma}

\section{Montesinos knots of genus two}\label{sec:Montg2}

%Since PCSC is true for any pretzel knot \cite{StipsiczSzabo}, by Lemma~\ref{lem:Hanselman}, one of next targets is a set of non-pretzel Montesinos knots of genus two. 
Since PCSC is true for any pretzel knot \cite{StipsiczSzabo} and any two-bridge knot \cite{IchiharaJongMattmanSaito}, by Lemma~\ref{lem:Hanselman}, one of the next targets is the set of non-pretzel and non-two-bridge Montesinos knots of genus two. 
Here we enumerate such knots. 
We follow the conventions of \cite{HirasawaMurasugi} for our notation to describe Montesinos knots. 
In particular, continued fractions are subtractive, that is, $\left[ c_1, c_2, \dots, c_n\right]$ indicates the following subtractive continued fraction: 
\[ \dfrac{1}{c_1 - \dfrac{1}{c_2 - \cdots -\dfrac{1}{c_n}}} \]
Subtractive continued fractions satisfy the following formulae, see \cite[p.\ 57 (3), (5), (6), (7), (8)]{HirasawaMurasugi}: 
\begin{align}
    [\, \ldots, c_{n-1}, \pm 2] &= [\, \ldots, c_{n-1} \mp 1, \mp 2], \label{eq:HM3} \\ 
    [\, \ldots, c_{i-1}, \pm 1, c_{i+1}, \ldots \,] &= [\, \ldots, c_{i-1} \mp 1, c_{i+1}\mp 1, \ldots \, ] \quad \text{if } i \ge 2, \label{eq:HM5} \\ 
    [\, \ldots, c_{n-1}, \pm 1 ] &= [\, \ldots, c_{n-1} \mp 1 ], \label{eq:HM6} \\ 
    [2,2,\dots, 2, c_{k+1}, \ldots \, ] &= 1 + [-(k+1), c_{k+1}-1, \ldots \, ] \quad \text{if } k \ge 1, \label{eq:HM7} \\ 
    [-2,-2,\dots, -2, c_{k+1}, \ldots \, ] &= -1 + [k+1, c_{k+1}+1, \ldots \, ] \quad \text{if } k \ge 1. \label{eq:HM8} 
\end{align}

\begin{proposition}\label{prop:Monteg2}
Let $a, b, c, d, e$ be non-zero integers. 
A non-pretzel and non-two-bridge Montesinos knot of genus two is equivalent to one of the following: 
\begin{description}
 \item[(o1)] 
$M\left([2a+1, 2b], [2c+1], [2d+1], [2e+1] \right) = M\left(\dfrac{2b}{4ab + 2b-1}, \dfrac{1}{2c+1}, \dfrac{1}{2d+1}, \dfrac{1}{2e+1} \right)$ \\ ($a, c, d, e \ne -1$)
\item[(o1')] 
$M\left([2a+1, 2b], [2c+1], [2d+1], \pm 1 \right) = M\left(\dfrac{2b}{4ab + 2b-1}, \dfrac{1}{2c+1}, \dfrac{1}{2d+1}, \pm 1\right)$ ($a, c, d \ne -1$)
\item[(o2)] 
$M\left([2a+1, 2b], [2c+1, 2d], [2e+1] \right) = M\left(\dfrac{2b}{4ab + 2b-1}, \dfrac{2d}{4cd + 2d-1}, \dfrac{1}{2e+1} \right)$ ($a,c,e \ne -1$)
\item[(o3)] 
$M\left([2a, \pm3], [2b+1] ,[2c+1] \right) = M\left(\dfrac{3}{6a \mp 1}, \dfrac{1}{2b+1} , \dfrac{1}{2c+1} \right)$ ($b,c \ne -1$ and $|a| \ge 2$)
\item[(o3')] 
$M\left([\pm2, \mp3], [2b+1] ,[2c+1] \right) = M\left( \pm \dfrac{3}{7}, \dfrac{1}{2b+1}, \dfrac{1}{2c+1} \right)$ ($b,c \ne -1$)
\item[(o4)] 
$M\left([2a, \pm2, 2b+1], [2c+1] ,[2d+1] \right) = M\left(\dfrac{4b +2 \mp 1}{8ab +4a \mp 2a \mp 2b \mp 1}, \dfrac{1}{2c+1} , \dfrac{1}{2d+1} \right)$  ($b,c,d \ne -1$ and $|a| \ge 2$)
\item[(o4')] 
$M\left([\pm2, \mp2, 2b+1], [2c+1] ,[2d+1] \right) = M\left(\dfrac{\pm 4b + 1 \pm 2}{10b + 5 \pm2}, \dfrac{1}{2c+1} , \dfrac{1}{2d+1} \right)$  ($b,c,d \ne -1$)
\item[(o5)] 
$M\left([2a+1, 2b, 2c], [2d+1] ,[2e+1] \right) = M\left(\dfrac{4bc-1}{8abc + 4bc -2a -2c -1}, \dfrac{1}{2d+1} ,\dfrac{1}{2e+1} \right)$  ($a,d,e \ne -1$)
\item[(e1)] 
$M\left([2a], [2b, 2c] ,[2d, 2e] \right) = M\left(\dfrac{1}{2a}, \dfrac{2c}{4bc-1}, \dfrac{2e}{4de-1} \right)$ 
\item[(e2)] 
$M\left([2], [-2, 2a] ,[2, 2b], [-2, 2c] \right) = M\left(\dfrac{1}{2}, \dfrac{-2a}{4a+1}, \dfrac{2b}{4b-1}, \dfrac{-2c}{4c+1} \right)$ or its mirror 
\item[(e3)] 
$M\left([3, 2a+1], [-3] ,[3], [-3] \right) = M\left(\dfrac{2a+1}{6a+2}, \dfrac{-1}{3}, \dfrac{1}{3}, \dfrac{-1}{3} \right)$ or its mirror ($a \ne 0,-1$) 
\end{description}
\end{proposition}
\begin{proof}
The genus of a Montesinos knot is completely determined by Hirasawa and Murasugi~\cite{HirasawaMurasugi}. 
Let 
\[ K= M\(\beta_1/\alpha_1, \beta_2/\alpha_2, \dots, \beta_r/\alpha_r \mid \gamma \)\] 
be a Montesinos knot\footnote{We use a notation $M\(\beta_1/\alpha_1, \dots, \beta_r/\alpha_r \mid \gamma \)$ instead of $M\(\beta_1/\alpha_1, \dots, \beta_r/\alpha_r \mid e \)$ since the symbol $e$ is used in the statement of Proposition~\ref{prop:Monteg2}.}. 
If $r \le 2$, then $K$ is a two-bridge knot. 
Thus, we assume that $r \ge 3$.  
We also assume that for all $i=1,2, \dots, r$, we have $\alpha_i >1$ and $- \alpha_i < \beta_i < \alpha_i$, with $\gcd (\alpha_i, \beta_i) =1$ as in \cite{HirasawaMurasugi}. 

First we consider the case where $K$ is of odd type, that is, suppose that $\alpha_i$ is odd for $i = 1,2,\cdots,r$. 
For each $i$, let 
\[ S_i = [2 {a_1}^{(i)}, {b_1}^{(i)}, 2 {a_2}^{(i)}, {b_2}^{(i)}, \dots, 2 {a_{q_i}}^{(i)}, {b_{q_i}}^{(i)}] \] 
be a strict continued fraction of $\beta_i/ \alpha_i$, that is, if $|{a_j}^{(i)}| = 1$, then ${a_j}^{(i)}{b_j}^{(i)} < 0$. 
By \cite[Theorem 3.1]{HirasawaMurasugi}, we have 
\[ g(K) = \(\sum_{i=1}^r b^{(i)} + |\gamma| -1\)/2 \, , \] 
where $b^{(i)} = \sum_{j=1}^{q_i} |{b_j}^{(i)}|$. 
Now we assume that $g(K) = 2$, then we have 
\[ \sum_{i=1}^r b^{(i)} + |\gamma| = 5 \, . \]
Since $b^{(i)} \ge 1$ for each $i$, we have $3 \le r \le 5$. 
Here we start the case by case argument. 

\begin{enumerate}[leftmargin=17pt]
\item 
Assume that $r = 5$. 
Then $\gamma = 0$, and $b^{(i)} = 1$, that is, $| b_1^{(i)} | = 1$ for $i=1,\dots,5$. 
By Equation~\eqref{eq:HM6}, we have 
\[ S_i = [2{a_1}^{(i)}, {b_1}^{(i)}] = [ 2{a_1}^{(i)} - {b_1}^{(i)} ]\, . \]
Note that $|2 {a_1}^{(i)} - {b_1}^{(i)}| \ge 3$ for each $i$ since if ${a_1}^{(i)} = \pm 1$, then ${b_1}^{(i)} = \mp 1$ by the definition of a strict continued fraction. 
Therefore $K$ is a pretzel knot with five strands; 
\[ P \( 2{a_1}^{(1)} - {b_1}^{(1)}, 2{a_1}^{(2)} - {b_1}^{(2)}, 2{a_1}^{(3)} - {b_1}^{(3)}, 2{a_1}^{(4)} - {b_1}^{(4)}, 2{a_1}^{(5)} - {b_1}^{(5)} \) . \] 
We exclude this case since we only enumerate non-pretzel Montesinos knots. 

\item 
Assume that $r = 4$. 
Then $|\gamma| = 0 $ or $1$. 
\begin{enumerate}[leftmargin=17pt]
\item
Assume that $|\gamma| = 1$. 
Then by the same argument as in (i), 
\[ K = P\( 2{a_1}^{(1)} - {b_1}^{(1)}, 2{a_1}^{(2)} - {b_1}^{(2)}, 2{a_1}^{(3)} - {b_1}^{(3)}, 2{a_1}^{(4)} - {b_1}^{(4)}, \gamma \). \] 
Since $|2{a_1}^{(i)} - {b_1}^{(i)}| \ge 3$, 
each of $\{1,-1\}, \{2,-1\}, \{1,-2\}$ is not a subset of $\{ 2{a_1}^{(1)} - {b_1}^{(1)}, 2{a_1}^{(2)} - {b_1}^{(2)}, 2{a_1}^{(3)} - {b_1}^{(3)}, 2{a_1}^{(4)} - {b_1}^{(4)}, \gamma  \}$. 
Thus, $K$ is also a pretzel knot with five strands, which have to be excluded. 

\item
Assume that $|\gamma| = 0$. 
By a cyclic permutation, we may assume that $b^{(1)} = 2$ and $b^{(2)} = b^{(3)} = b^{(4)} = 1$. 
Then we have $S_i = [2{a_1}^{(i)} - {b_1}^{(i)}]$ for $i = 2, 3, 4$. 
There are two cases for the condition $b^{(1)} = 2$, that is, the case where $|{b_1}^{(1)}| = |{b_2}^{(1)}| = 1$, or the case where $|{b_1}^{(1)}| = 2$. 

First we consider the case where $|{b_1}^{(1)}| = |{b_2}^{(1)}| = 1$. 
Then, by Equations~\eqref{eq:HM5} and \eqref{eq:HM6}, we have 
%by setting ${a_i}^{(1)} = a_i$ and ${b_i}^{(1)} = b_i$ ($|b_i| =1$), we have 
% \begin{align*}
\[ S_1 
= [2{a_1}^{(1)}, {b_1}^{(1)}, 2{a_2}^{(1)}, {b_2}^{(1)}] 
% &= [2{a_1}, b_1, 2{a_2}, b_2] \\ 
% &= [2{a_1}, b_1, 2{a_2}- b_2] \\ 
% &= [2{a_1}- b_1, 2{a_2} - b_1 - b_2] \, . 
%&= [2{a_1}^{(1)}, {b_1}^{(1)}, 2{a_2}^{(1)}- {b_2}^{(1)}] \\ 
= [2{a_1}^{(1)}- {b_1}^{(1)}, 2{a_2}^{(1)} - {b_1}^{(1)} - {b_2}^{(1)}] \, . 
\] 
% \end{align*}
Setting $2{a_1}^{(1)}- {b_1}^{(1)} = 2a +1$, 
$2{a_2}^{(1)} - {b_1}^{(1)} - {b_2}^{(1)}= 2b$, 
$2a_1^{(2)} - b_1^{(2)} = 2c+1$, 
$2a_1^{(3)} - b_1^{(3)} = 2d+1$, and 
$2a_1^{(4)} - b_1^{(4)} = 2e+1$, 
we have the knots 
\[ M\left([2a+1, 2b], [2c+1], [2d+1], [2e+1] \right), \] 
where $a, b, c, d , e$ are non-zero integers with $a,c,d,e \ne -1$. 
All of these knots form the family (o1). 
% Further, if $a \ge 1$, then $b \ne 1$ holds and if $a \le -2$, then $b \ne -1$ holds in this situation. 
% with the conditions $a, b, c, d, e \ne 0$ and $a,c,d,e \ne -1$. 
% Notice that $a, b \ne 0$ since $|2 a_i - \varepsilon_i| \ge 3$ for $i=1,2$, 
% and that $c, d, e \ne 0$
Next we consider the case where $|{b_1}^{(1)}| = 2$. 
By Equation~\eqref{eq:HM3}, we have %, by setting ${a_i}^{(1)} = a_i$ and ${b_i}^{(1)} = \varepsilon_i$, 
\[ S_1 
= [2{a_1}^{(1)}, {b_1}^{(1)}]  
= [2{a_1}^{(1)} - {b_1}^{(1)}/2, -{b_1}^{(1)}] \, . 
%= [2{a_1}^{(1)} \mp 1, \pm 2] \, . 
\] 
Then we obtain knots already appeared in the previous case, that is, these knots are contained in the family (o1). 
\end{enumerate}

\item 
Assume that $r = 3$. 
Then $|\gamma| = 0, 1,$ or $2$. 
\begin{enumerate}[leftmargin=17pt]
\item
Assume that $|\gamma| = 2$. 
Then by the same argument as in (i), 
\[ K = P \(2{a_1}^{(1)} - {b_1}^{(1)}, 2{a_1}^{(2)} - {b_1}^{(2)}, 2{a_1}^{(3)} - {b_1}^{(3)}, \varepsilon, \varepsilon\), \]  
% Since $|2a^{(i)} - \varepsilon_i| \ge 3$, 
% each of $\{1,-1\}, \{2,-1\}, \{1,-2\}$ is not a subset of $\{ 2a^{(1)} - \varepsilon_1, 2a^{(2)} - \varepsilon_2, 2a^{(3)} - \varepsilon_3, \varepsilon, \varepsilon \}$. 
% Thus, 
where $|\varepsilon| = 1$. 
Then $K$ is a pretzel knot with five strands, which have to be excluded. 

\item[(iii-2)]
Assume that $|\gamma| = 1$. 
% By a cyclic permutation, we may assume that $b^{(1)} = 2$ and $b^{(2)} = b^{(3)} = 1$. 
By the same argument as in (ii-2), we have the family (o1'). 
%  admitting $e = 0, -1$. 
% Then combining the case (ii-2), we obtain all knots contained in the family (o1). 

\item[(iii-3)]
Assume that $|\gamma| = 0$. 
By a cyclic permutation, we may assume that either 
$b^{(1)} = b^{(2)} = 2$ and $b^{(3)} = 1$, 
or $b^{(1)} = 3$ and $b^{(2)} = b^{(3)} = 1$. 
If $b^{(1)} = b^{(2)} = 2$ and $b^{(3)} = 1$, then by the same argument as in (ii-2) or (iii-2), we obtain the family (o2). 
If $b^{(1)} = 3$ and $b^{(2)} = b^{(3)} = 1$, then we have the families (o3)--(o5) as follows: 
Since $b^{(1)} = 3$, there are four cases for $S_1$ as 
\begin{align*}
S_1 
&= \begin{cases}
[2 {a_1}, 3 \varepsilon_1], \\ 
[2 a_1, 2 \varepsilon_1, 2 a_2, \varepsilon_2], \\ 
[2 a_1,\varepsilon_1, 2 a_2, 2\varepsilon_2], \\ 
[2 a_1, \varepsilon_1, 2 a_2, \varepsilon_2, 2 a_3, \varepsilon_3], 
\end{cases}\\
&= \begin{cases}
[2 a_1, 3 \varepsilon_1],  \\ 
[2 a_1, 2 \varepsilon_1, 2 a_2 - \varepsilon_2],  \\ 
[2 a_1 - \varepsilon_1, 2 a_2 - \varepsilon_1 - \varepsilon_2, -2\varepsilon_2], \\ 
[2 a_1 - \varepsilon_1, 2 a_2 - \varepsilon_1 - \varepsilon_2, 2 a_3 - \varepsilon_2 - \varepsilon_3], 
\end{cases} 
\end{align*}
where $a_i = {a_i}^{(1)}$, and $|\varepsilon_i| = 1$. 

First we consider the case where $S_1 = [2 a_1, 3 \varepsilon_1]$. 
If $a_1 = \pm 1$, then $\varepsilon_1 = \mp 1$, and then we obtain the family (o3'). 
If $a_1 \ne \pm 1$, then we obtain the family (o3). 

Next we consider the case where $S_1 = [2 a_1, 2 \varepsilon_1, 2 a_2 - \varepsilon_2]$. 
By the same argument as in (o3') and (o3), we have the families (o4) and (o4'). 

Finally, we obtain the family (o5) from the cases where $S_1 = [2 a_1 - \varepsilon_1, 2 a_2 - \varepsilon_1 - \varepsilon_2, -2\varepsilon_2]$ or $S_1 = [2 a_1 - \varepsilon_1, 2 a_2 - \varepsilon_1 - \varepsilon_2, 2 a_3 - \varepsilon_2 - \varepsilon_3]$. 
% if $a_2 \ge 1$, then $\varepsilon_2 = -1$ and $2 a_2 - \varepsilon_1 - \varepsilon_2 = 2 a_2 + 1 - \varepsilon_1$ is a positive even integer. 
% Further, $2 a_3 - \varepsilon_2 - \varepsilon_3 = 2 a_3 +1 - \varepsilon_3$ is an even integer without $0$ and $2$. 
% Similarly, if $a_2 \le 1$, then $\varepsilon_2 = 1$ and $2 a_2 - \varepsilon_1 - \varepsilon_2 = 2 a_2 -1 - \varepsilon_1$ is a negative even integer. 
% Further, $2 a_3 - \varepsilon_2 - \varepsilon_3 = 2 a_3 -1 - \varepsilon_3$ is an even integer without $0$ and $-2$. 
% Thus, in this case, we have 
% \begin{equation}\label{eq-o52}
% S_1 = 
% \begin{cases}
%    [2 a_1 - \varepsilon_1, 2p, 2t],  \\ 
%    [2 a_1 - \varepsilon_1, -2p, -2t], \\ 
% \end{cases}
% \end{equation}
% where $p$ is a positive integer and $t$ is an integer with $t\ne 0,1$. 
% Combining \eqref{eq-o51} and \eqref{eq-o52}, we have 
% \[ 
%    S_1 = 
% \begin{cases}
%    [2 a_1 - \varepsilon_1, 2p, 2t],  \\ 
%    [2 a_1 - \varepsilon_1, -2p, -2t], \\ 
% \end{cases}
% \] 
% with $p \ge 1$ and $t \ne 0$. 
% Then we have the family (o5). 
\end{enumerate}
\end{enumerate}
Now we complete the enumeration for odd type. 

Next we consider Montesinos knots of even type. 
Suppose that $\alpha_1$ is even and $\beta_1$ is odd, and that $\alpha_i$ is odd and $\beta_i$ is even for any $i \ge 2$. 
% Suppose also that $\gamma$ is even. 
For each $i$, let $S_i = [2 {c_1}^{(i)}, 2 {c_2}^{(i)}, \dots, 2 {c_{m_i}}^{(i)}]$ be an even continued fraction of $\beta_i/ \alpha_i$. 
Note that $m_1$ is odd and $m_i$ is even for $i=2,\dots, r$. 
\begin{enumerate}[leftmargin=17pt]
\item 
Assume that $\gamma \ne 0$. 
By \cite[Theorem 3.2]{HirasawaMurasugi}, $\displaystyle g(K) = \dfrac12\(1 + \sum_{i=1}^r m_i\)$. 
Thus, if $g(K) = 2$, then we have \[ \displaystyle \sum_{i=1}^r m_i = 3 , \] which contradicts to the conditions $m_1 \ge 1$ and $m_i \ge 2$ for $i = 2, 3, \dots, r$, and $r \ge 3$. 
So we may assume that $\gamma = 0$. 

\item 
Assume that $\gamma =0$. 
\begin{enumerate}[leftmargin=17pt]
\item 
Assume that $\(c_1^{(1)}, c_1^{(2)}, \dots, c_1^{(r)}\) \ne \pm \(1, -1, \dots, 1, -1\)$. 
By \cite[Theorem 3.2 (II)]{HirasawaMurasugi}, 
\[ \displaystyle g(K) = \dfrac12\(-1 + \sum_{i=1}^r m_i\). \] 
Thus, if $g(K) = 2$, then we have 
\[ \sum_{i=1}^r m_i = 5 \, . \]
Since each $m_1 \ge 1$ and $m_i \ge 2$ for any $i \ge 2$, we have $r=3$ and $m_1 =1$, $m_2 = m_3 = 2$. 
Then we have the family (e1). 
\item 
Assume that $\(c_1^{(1)}, c_1^{(2)}, \dots, c_1^{(r)}\) = \pm \(1, -1, \dots, 1, -1\)$. 
By taking the mirror image if necessary, we may assume that $\(c_1^{(1)}, c_1^{(2)}, \dots, c_1^{(r)}\) = \(1, -1, \dots, 1, -1\)$. 
Note that $r$ is even in this case. 
Let $p_i$ be the number of leading 2's in $S_i$, if $i$ is odd, or the number of leading $-2$'s in $S_i$, if $i$ is even. 
Let $p = \min \set{ p_1, \dots, p_r }$. 
Then, by \cite[Theorem 3.2 (III)]{HirasawaMurasugi}, $\displaystyle g(K) = \dfrac12\(1 + \sum_{i=1}^r m_i\) -(p+1)$. 
Thus, if $g(K) = 2$, then we have 
\[ \sum_{i=1}^r m_i = 2p+ 5 \, . \]
Since $p \le m_i$ for each $i = 1,2, \dots, r$, we have $rp \le 2p+5$. 
Thus, we have $r \le 2 + 5/p$, and thus, $r \le 7$. 
Furthermore, since $r \ge 3$ and $r$ is even, $r = 4$ or $6$. 

If $r = 6$, then we have $p=1$ and $\sum_{i=1}^6 m_i = 7$, which contradicts to the conditions $m_1 \ge 1$ and $m_i \ge 2$ for $i = 2,\dots, 6$. 

If $r = 4$, then we have $p=1$ or $2$.
If $p = 1$, then $\sum_{i=1}^4 m_i = 7$. 
Since $m_1 \ge 1$ and $m_2,m_3,m_4 \ge 2$, we have $m_1 =1$ and $m_2 = m_3 = m_4 = 2$. 
Then we have the family (e2). 
If $p = 2$, then $\sum_{i=1}^4 m_i = 9$. 
The case where $m_1 = 1$ is unsuitable since $p = 2$ and $m_1 \ge p$. 
Thus, we have $m_1 = 3$ and $m_2 = m_3 = m_4 = 2$ (recall that $m_1$ is odd and $m_i$ is even for $i=2,3,4$). 
Since $p = 2$, we have 
\[ S_1 = [2,2,2a], \quad S_2 = [-2, -2], \quad S_3 = [2, 2], \quad S_4 = [-2, -2], \]
where $a \ne 0, 1$. 
By Equations~\eqref{eq:HM7} and \eqref{eq:HM8}, these can be deformed into the forms 
\[ S_1 = 1 + [-3,2a-1], \quad S_2 = -1 + [3], \quad S_3 = 1 + [-3], \quad S_4 = -1 + [3]. \]
Then, the Montesinos knot $M (1 + [-3, 2a - 1], -1 + [3], 1 + [-3], -1 + [3])$ is modified by using flype moves to the Montesinos knot $M ([-3, 2a - 1], [3], [-3], [3])$ as shown in Figure~\ref{fig:e3}, which consists of the family (e3).
% By using flype moves, we see that 
% \[ M\(1 + [-3,2a-1], -1 + [3], 1 + [-3], -1 + [3]\) = M\([-3,2a-1], [3], [-3], [3]\) \] 
% as shown in Figure~\ref{fig:e3}. 
% Then we have the family (e3). 
% Note that 
% \begin{align*}
% S_1 &= 1 + [-3,2a-1] = 1 + \dfrac{1}{-3 - \dfrac{1}{2a-1}} = \dfrac{4a-1}{6a-2}, \\  
% S_2 = S_4 &= -1 + [3] = -1 + \dfrac{1}{3} = -\dfrac{2}{3}, \quad S_3 = -S_2 = \dfrac{2}{3}.     
% \end{align*}
\begin{figure}[htb]
    \centering
	\begin{overpic}[width=.4\textwidth]{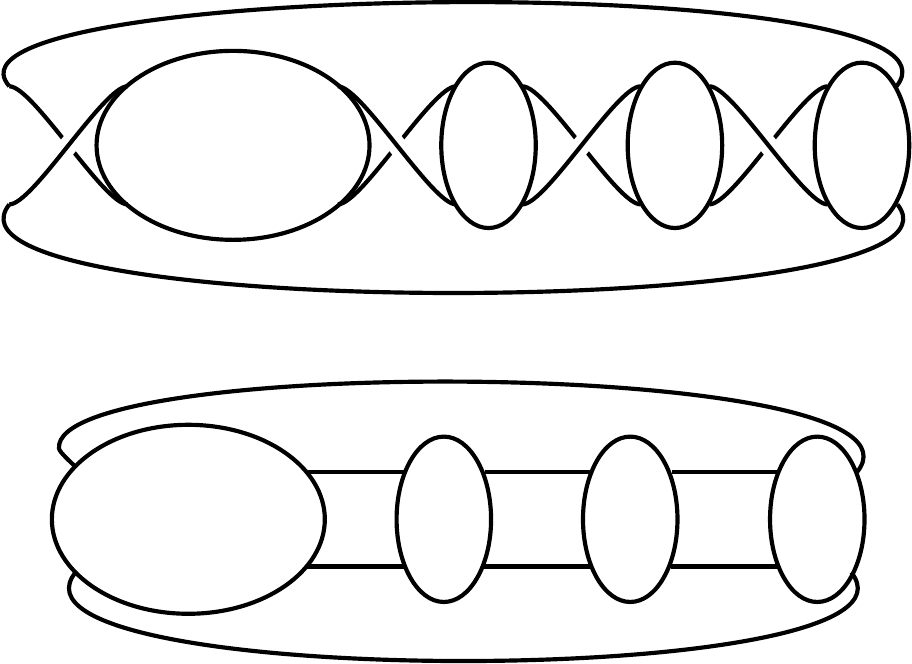}
 		\put(12,55.5){$[-3,2a-1]$}
    	\put(51,55.5){$[3]$}
        \put(71.5,55.5){$[3]$}
        \put(92,55.5){$[3]$}
		\put(50,34){\rotatebox{90}{$=$}}
 		\put(7,14.3){$[-3,2a-1]$}
    	\put(46,14.3){$[3]$}
        \put(66.5,14.3){$[3]$}
        \put(87,14.3){$[3]$}
	\end{overpic}
	\caption{$M\(1 + [-3,2a-1], -1 + [3], 1 + [-3], -1 + [3]\) = M\([-3,2a-1], [3], [-3], [3]\)$}
	\label{fig:e3}
\end{figure}
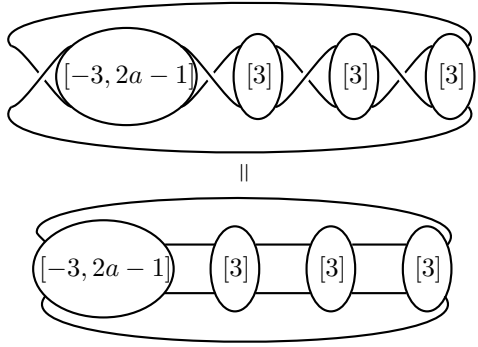

Now we complete the proof of Proposition~\ref{prop:Monteg2}. \qedhere 
\end{enumerate}
\end{enumerate}
\end{proof}

\section{Proofs of Theorems}\label{sec:proof}

\subsection{Proof of Theorem~\ref{thm:alt} (PCSC for alternating knots)}
% We show that PCSC is true for all alternating knots except for two families listed in Theorem~\ref{thm:alt}. 
Here we give a proof of Theorem~\ref{thm:alt}. 

% \begin{theorem}\label{thm:alt}
% If an alternating knot $K$ admit purely cosmetic surgeries, then $K$ is equivalent to one the knots depicted in Figure~\ref{}. 
% \end{theorem}
\begin{proof}[Proof of Theorem~\ref{thm:alt}]
Assume that a non-trivial alternating knot $K$ admits purely cosmetic surgeries. 
Since composite knots admit no purely cosmetic surgeries~\cite{TaoComposite}, $K$ is prime. 
By Lemma~\ref{lem:Hanselman} and Equation~\eqref{eq:tau-sig}, we have $g(K) = 2$ and $\sigma(K) = 0$. 
Then $K$ is obtained from one of the alternating knot diagrams depicted in Figure~\ref{fig:G2} by applying $\overline{t_{2}^\prime}$~moves (Figure~\ref{fig:t2move}) up to taking the mirror image~\cite[Proposition 3.2]{Stoimenow-g2}. 
If $K$ is obtained from one of $8_{12}$, $10_{58}$, $12_{1202}$ by $\overline{t_{2}^\prime}$~moves, then $a_2(K) \ne 0$ holds\footnote{It can be confirmed by a direct calculation of $a_2$ using the skein relation~\eqref{eq:a2skein}.}, which contradicts to Lemma~\ref{lem:IchiharaWu}. 
More precisely, for an alternating knot $K$ with $g(K) =2$, 
\[a_2(K) \ne 0 \iff [\Delta]_1 \ne 4[\Delta]_0\, ,  \]
where $[\Delta]_i$ denotes the absolute value of the $i$-th coefficient of the normalized Alexander polynomial for $K$, see \cite{Jong}. 
By the study on the Alexander polynomials of alternating knots of genus two, in particular, by \cite[Table 1]{Jong}\footnote{\cite[Table 1]{Jong} contains an error. For $10_{58}$, $m=4$ and $M=5$ are correct.}, we see that if $K$ is obtained from one of the diagrams $8_{12}$, $10_{58}$, $12_{1202}$, then $a_2(K) \ne 0$. 
Thus, $K$ is obtained from one of $6_3$, $7_7$, $9_{41}$. 

Assume that $K$ is obtained from $7_7$. 
Then $K$ is equivalent to the knot 
\[ K(a,b,c,d,e) = M( [2a+1, -2b], [2c+1, -2d], [2e+1] ) \] 
depicted in Figure~\ref{fig:7_7}, which is an alternating Montesinos knot. 
Here $a,b,c,d,e$ are integers with $a,c,e \ge 0$ and $b, d \ge 1$. 
\begin{claim}\label{clm:a2w3} 
	For $K = M\left([2a+1, 2b], [2c+1, 2d], [2e+1] \right)$, 
	% ($a,c,e \ne 0,-1$, $b,c \ne 0$), 
	we have 
	\begin{align*} 
		a_2(K) &= d (c + e + 1) + b (a + d + e + 1) \, ,  \\ 
		w_3(K) &= (d (c + e + 1) (c + d + e + 1) + 2 b d (1 + c + e) + b (a + d + e + 1) (a + b + d + e + 1))/4 \, . 
	\end{align*} 
\end{claim} 
\begin{proof}[Proof of Claim~\ref{clm:a2w3}] 
	First we calculate $a_2$. 
	Applying Equation~\eqref{eq:a2skein} to any crossing in the twist box labeled $2b$, we have 
	\[ a_2(K(a,b,c,d,e)) = a_2(K(a,b-1,c,d,e)) - \(-(a+d+e+1)\)\, . \]
	Repeating this procedure $|b|$ times, we have
	\[a_2(K(a,b,c,d,e)) = a_2(K(a,0,b,c,d)) - b \(-(a+d+e+1)\)\, . \]
	Since the knot $K(a,0,b,c,d)$ is equivalent to the knot $\mathrm{DT}(2(c+e+1),2d)$, by Lemma~\ref{lem:a2TK}, we have 
	\[ a_2(K(a,b,c,d,e)) = (c+e+1)d + b(a+d+e+1) \, . \]
	Next we calculate $w_3$. 
	Notice that the two-component link obtained by smoothing any crossing in the twist box labeled $2b$ consists of two unknots. 
	Thus, we can apply Lemma~\ref{lem:w3twists} to the $b$-full twists. 
	Using Lemma~\ref{lem:w3TK}, we have 
	\begin{align*}
		w_3(K(a,b,c,d,e)) &= w_3(K(a,0,c,d,e)) + \dfrac{b}{2} a_2(K(a,0,c,d,e)) + \dfrac{b}{4} \lk (\lk - b) \\
		&= w_3(\mathrm{DT}(2(c+e+1),2d)) + \dfrac{b}{2} a_2(\mathrm{DT}(2(c+e+1),2d)) \\  
		&\quad + \dfrac{b}{4} \(-(a+d+e+1) \cdot \(-(a+d+e+1+b)\)\) \\ 
		&= \dfrac{d}{4}(c+e+1)(c+d+e+1) + \dfrac{bd}{2}(c+e+1) \\ 
		&\quad + \dfrac{b}{4}(a+d+e+1)(a+b+d+e+1) \, . 
	\end{align*}
	This completes the proof of Claim~\ref{clm:a2w3}. 
\end{proof}

\begin{figure}[htb]
	\centering
	\includegraphics[width=.5\textwidth]{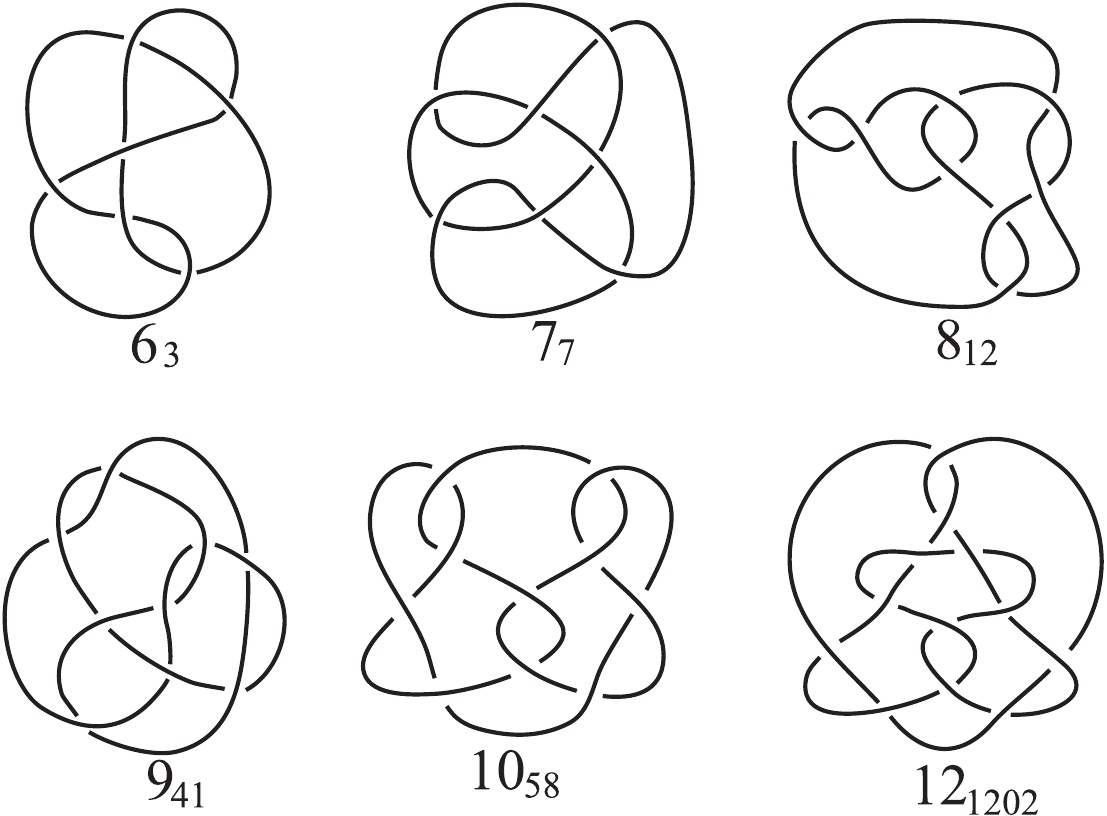}
	\caption{Generators of alternating knots of genus two with $\sigma = 0$}
	\label{fig:G2}
\end{figure}
\begin{figure}[htb]
    \centering
	\begin{overpic}[width=.32\textwidth]{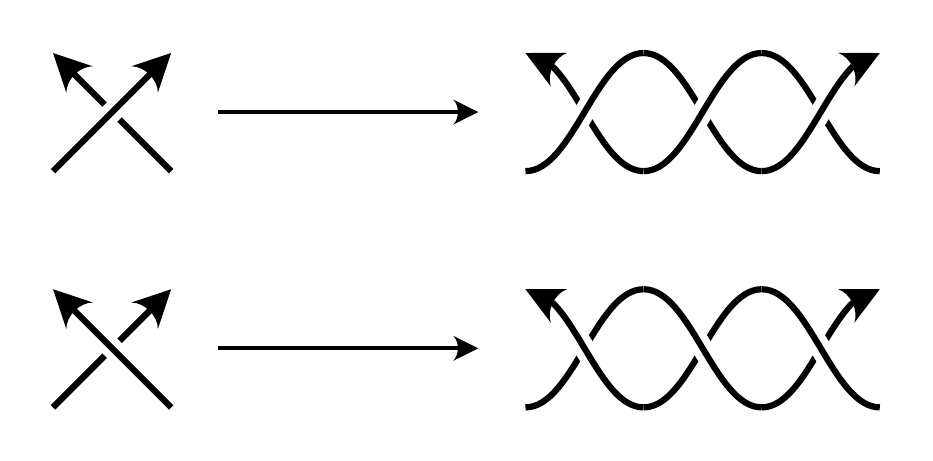}
		\put(25,40.7){$\overline{t_{2}^\prime}$~move}
		\put(25,16){$\overline{t_{2}^\prime}$~move}
	\end{overpic}
	\caption{{$\overline{t_{2}^\prime}$~move}}
	\label{fig:t2move}
\end{figure}
\begin{figure}[htb]
    \centering
	\begin{overpic}[width=.27\textwidth]{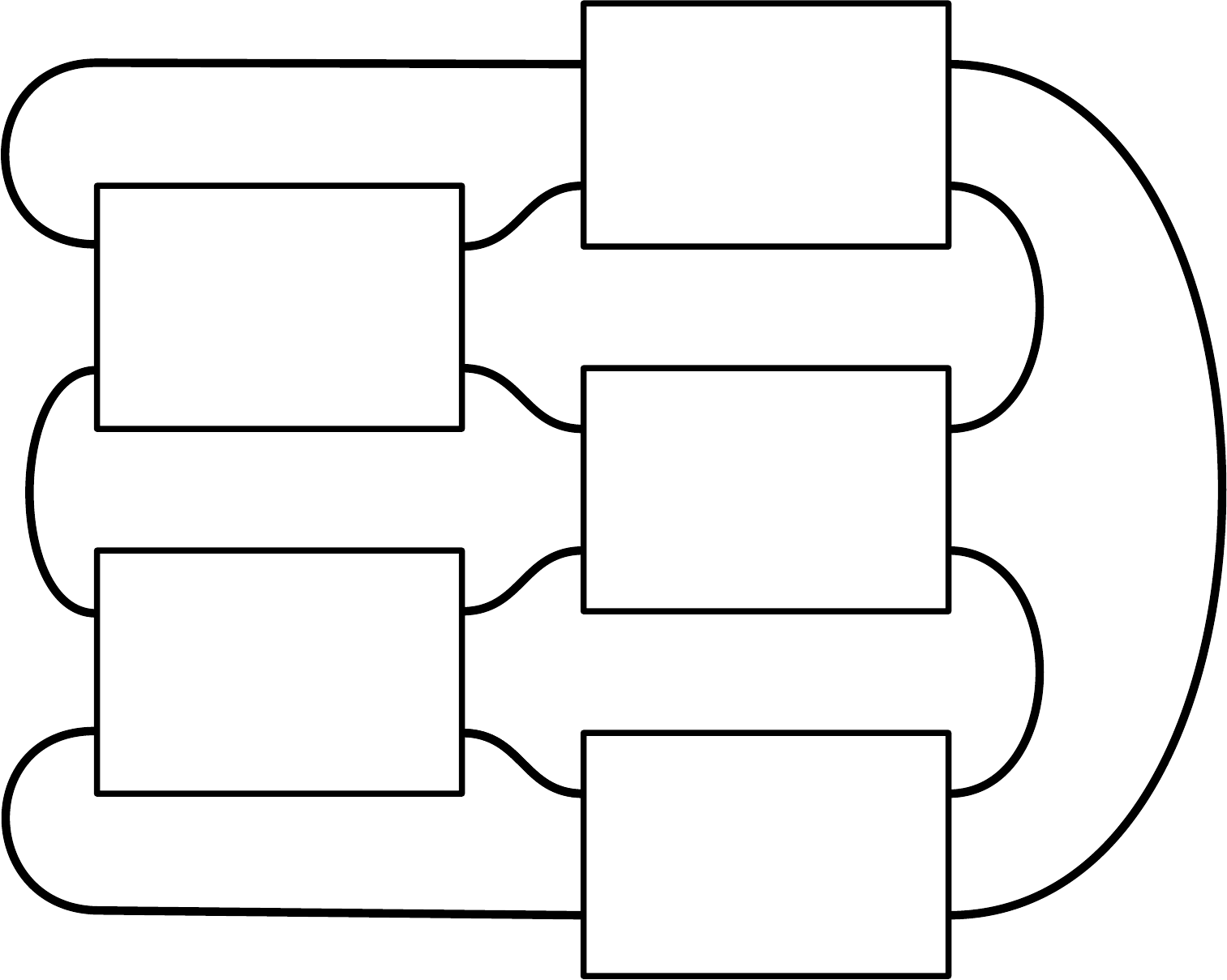}
		\put(14,52){$-2b$}
		\put(14,22){$-2d$}
		\put(51,67){$2a+1$}
		\put(51,37){$2e+1$} 
		\put(51,8){$2c+1$} 
	\end{overpic}
	\caption{$K(a,b,c,d,e)$}
	\label{fig:7_7}
\end{figure}
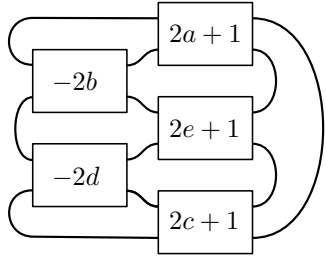

\begin{claim}\label{clm:o2} 
	Alternating Montesinos knots $M\left([2a+1, -2b], [2c+1, -2d], [2e+1] \right)$ $(a, c, e \ge 0$ and $b,d \ge 1)$ admit no purely cosmetic surgeries. 
\end{claim} 
\begin{proof} 
By Claim~\ref{clm:a2w3}, we have 
\begin{align} 
	a_2(K) &= -d (c + e + 1)  -b (a - d + e + 1) \label{eq:a2}
\end{align}
and 
\begin{align}
	w_3(K) &= (-d (c + e + 1) (c - d + e + 1) + 2 b d (1 + c + e) - b (a - d + e + 1) (a - b - d + e + 1))/4 \label{eq:w3} \, . 
\end{align} 
Assume that $a_2(K) =0$. 
Then we have $c = (-b - a b - d + b d - b e - d e)/d$. 
Substituting this for Equation~\eqref{eq:w3}, we have 
\begin{equation}\label{eq:w3-2}
    w_3(K) = - b (b + d) (1 + a + e) (1 + a - d + e)/(4 d).  
\end{equation}
Assume that $w_3(K) = 0$. 
Then by Equation~\eqref{eq:w3-2} and the conditions $a, c, e \ge 0$ and $b, d \ge 1$, we have 
\[ d = 1+a +e \, . \]
Substituting this for Equation~\eqref{eq:a2}, we have 
\[ -(1 + a + e) (c + e + 1) =0 \, , \]
which contradicts to that $a,c,e \ge 0$. 
This completes the proof of Claim~\ref{clm:o2}. 
\end{proof} 
The knots in Figure~\ref{fig:alt} coincide with the knots obtained from $6_3$ or $9_{41}$ by applying $\overline{t_{2}^\prime}$~moves. 
Now we complete the proof of Theorem~\ref{thm:alt}. 
\end{proof}

%We remark that the knots in Figure~\ref{fig:alt} contains infinitely many knots with $a_2 = 0$ and $w_3 = 0$. 
We remark that the knots in Figure~\ref{fig:alt} contain infinitely many knots with $a_2 = 0$ and $w_3 = 0$. 
For a knot $K(a,b,c,d,e,f)$ in the left side of Figure~\ref{fig:alt}, we have 
\begin{align*}
    a_2(K(a,b,c,d,e,f)) &= 1 + a (1 + b + c) + d + d e + b (c - e - f) + d f + e f - c (e + f), \\ 
    w_3(K(a,b,c,d,e,f)) &= \dfrac{1}{4}(2 c + a^2 (1 + b + c) - d + 2 c d - d^2 - 2 e - c^2 e - 2 d e + 2 c d e - d^2 e + c e^2 - d e^2 \\ 
        & \quad + b^2 (c - e - f)  - (2 + c^2 + d (2 + d) + 4 d e + e^2 - 2 c (d + 2 e)) f + (c - d - e) f^2 \\ 
        & \quad + a (b^2 + 4 b c + (1 + c) (1 + c - 2 e - 2 f) - 2 b (-1 + e + f)) \\ 
        & \quad + b (2 + c^2 + e^2 + 4 e f + f^2 - 4 c (e + f) + 2 d (1 + e + f))). 
\end{align*}
In this case, for any non-negative integer $a$, we have 
\[ a_2( K(a,a+1,0,0,a+1,0)) = w_3(K(a,a+1,0,0,a+1,0)) = 0.\]

For a knot $K$ in the right side of Figure~\ref{fig:alt}, we have 
\begin{align*}
    a_2(K(a,b,c,d,e,f)) &= c(1-e-f) + a (1 + b + c - d - f) + b (1 + c - d - e) - 2 (d + e + f), \\ 
    w_3(K(a,b,c,d,e,f)) &= \dfrac{1}{4}(-2 + c + c^2 - 6 d - 4 c d + 2 d^2 + a^2 (1 + b + c - d - f) - 6 f(1+c) - c^2 f + 2 d f \\
        & \quad + f^2(2+c) + b^2 (1 + c - d - e) - ce (6 + c - 2 f) + 2e (-3 + d + f)  + e^2(2 + c) \\ 
        & \quad + a (1 + b^2 + 4 b c + c^2 - 6 d - 6 f + (d + f)^2 - 4 e - 2 b (-2 + 2 d + e + f)  \\ 
        & \quad - 2 c (-2 + d + e + 2 f)) + b (1 + c^2 - 6 d  - 6 e - 4 f + (d + e)^2 - 2 c (-2 + d + 2 e + f))). 
    % a_2(K) &= c(1-e-f) + a (1 + b + c - d - e) + b (1 + c - d - f) - 2 (d + e + f), \\ 
    % w_3(K) &= \dfrac{1}{4}(-2 + c + c^2 - 6 d - 4 c d + 2 d^2 + a^2 (1 + b + c - d - e) - 6 e(1+c) - c^2 e + 2 d e \\
    %     & \quad + e^2(2+c) + b^2 (1 + c - d - f) - cf (6 + c - 2 e) + 2f (-3 + d + e)  + f^2(2 + c) \\ 
    %     & \quad + a (1 + b^2 + 4 b c + c^2 - 6 d - 6 e + (d + e)^2 - 4 f - 2 b (-2 + 2 d + e + f) - 2 c (-2 + d + 2 e + f)) \\ 
    %     & \quad + b (1 + c^2 - 6 d - 4 e - 6 f + (d + f)^2 - 2 c (-2 + d + e + 2 f))). 
\end{align*}
One of strategies to attack these knots is using finite type invariant obstructions of higher order as studied by Ito~\cite{ItoTetsuya}.

% \begin{lemma}[\cite{Stoimenow-g2}]\label{lem:Stoimenow-g2}
% Any alternating knot of genus two possesses a diagram which is obtained by $\overline{t_{2}^\prime}$~moves and flypes from one of the diagrams in Figure~\ref{fig:G2} up to taking the mirror image. 
% \end{lemma}

\subsection{Proof of Theorem~\ref{thm:Montesinos} (PCSC for Montesinos knots)}

\begin{proof}[Proof of Theorem~\ref{thm:Montesinos}]
Assume that a non-trivial Montesinos knot $K$ admits purely cosmetic surgeries. 
By Lemma~\ref{lem:Hanselman}, we have $g(K) = 2$. 
Then, since $K$ is neither a pretzel knot nor a two-bridge knot by \cite{StipsiczSzabo} nor \cite{IchiharaJongMattmanSaito}, $K$ is contained in the list of Proposition~\ref{prop:Monteg2}. 

Suppose that $K$ is contained in (e2) or (e3). 
Let $a, b, c, d, e$ be non-zero integers. 

\begin{claim}\label{clm:e2}
	For $K = K(a,b,c) = M\left([2], [-2,2a] ,[2,-2b], [-2,2c] \right)$, $a_2(K) = -2b$. 
\end{claim}
\begin{proof}
	Applying the skein relation~\eqref{eq:a2skein} to the full-twist of the first tangle, 
	\begin{align*}
		a_2\(K(a,b,c)\) 
		&= a_2\( \mathrm{DT}(-2,2a) \sharp \mathrm{DT}(2,-2b) \sharp \mathrm{DT}(-2,2c) \) - (-a+b-c) \\ 
		&= -a -b -c +a -b +c \\
		&= -2b. 
	\end{align*}
    Here $\sharp$ denotes the connected sum of knots. 
    Note that $a_2$ is additive under the connected sum. 
\end{proof}

\begin{claim}\label{clm:e3}
	For $K = K(a) = M\left([3, 2a+1], [-3] ,[3], [-3] \right)$, $a_2(K) = 2$. 
\end{claim}
\begin{proof}
	Applying the skein relation~\eqref{eq:a2skein} to the twist box labeled $2a+1$, 
	we have 
	\[ a_2(K(a)) = a_2(K(a-1)) . \]
	By a direct calculation, we have $a_2(K(0)) = 2$. 
\end{proof}

Then, by Claims~\ref{clm:e2} and \ref{clm:e3}, $a_2(K) \ne 0$ which contradicts to Lemma~\ref{lem:IchiharaWu2}. 
% Suppose that $K$ is contained in (o3). 
% Set RHS's of Equations~\eqref{eq:o3a2} and \eqref{eq:o3w3} to $0$ and solve them simultaneously, by using Mathematica~\cite{Mathematica}, we see that only $a=b=c=0$ is the solution. 
% Thus, we see that $a_2(K) \ne 0$ or $w_3(K) \ne 0$ which contradicts to Lemma~\ref{lem:IchiharaWu2}. 
This completes the proof of Theorem~\ref{thm:Montesinos}. 
\end{proof}

\subsection{Proof of Theorem~\ref{thm:altMontesinos} (PCSC for alternating Montesinos knots)}

The following lemma is useful to study alternating Montesinos knots, which follows from \cite[Theorem 10]{LickorishThistlethwaite} and \cite[Theorem B]{Murasugi87}. 

\begin{lemma}\label{lem:LT}
	Let $K$ be an alternating Montesinos knot. 
	Then a reduced Montesinos diagram of $K$ is an alternating diagram. In particular, all the rational tangles in $K$ have the same sign. 
\end{lemma}

However, to consider alternating Montesinos knots of genus two, it is not enough to consider the Montesinos diagrams described in Theorem~\ref{thm:Montesinos} to be alternating. 
For example, for a knot contained in (o1), if $a,c,d,e \ge 1$ and $b \le -1$, then the Montesinos diagram is alternating. 
On the other hand, if $a,b,c,d,e \ge 1$, then the Montesinos diagram is not alternating, although the knot admits an alternating diagram (obtained by modifying the diagram in the rational tangle $[2a+1,2b]$). 

% Here we prove Theorem~\ref{thm:altMontesinos}. 

\begin{proof}[Proof of Theorem~\ref{thm:altMontesinos}]
Let $K$ be an alternating Montesinos knot of genus two. 
Recall that if $\sigma(K) \ne 0$, then $K$ admits no purely cosmetic surgeries (cf.\ Equation~\eqref{eq:tau-sig}). 
We proceed with the proof by dividing the cases according to which family $K$ is included in. 
% Note that it is enough to consider the families (e1), (o1), (o1'), (o2), (o3), (o4), and (o5) listed in Proposition~\ref{prop:Monteg2}. 
Note that it is enough to consider the families (o1), (o1'), (o2), (o3), (o3'), (o4), (o4'), (o5) and (e1) listed in Proposition~\ref{prop:Monteg2}.

\begin{description}[leftmargin=30pt]
\item[(o1)]
Suppose that $K$ is contained in the family (o1). 
By taking a mirror image if necessary, we may assume that $a \ge 1$. 
Then $c,d,e \ge 1$ by Lemma~\ref{lem:LT}. 
\begin{enumerate}[leftmargin=17pt]
	\item 
	Suppose that $b \le -1$. 
	Then the diagram in Figure~\ref{fig:o1} becomes alternating. 
	By Lemma~\ref{lem:signature}, 
	\[ \sigma(K) = (2 |b| -1 +4) - 2|b| -1 = 2 \, . \]
	\item 
	Suppose that $b \ge 1$. 
	Then the diagram in Figure~\ref{fig:o1} is a negative diagram, that is, all crossings are negative. 
    Then, $\sigma(K) >0$ holds~\cite{Przytycki89,Traczyk88}. 
\end{enumerate}
\begin{figure}[htb]
    \centering
	\begin{overpic}[width=.35\textwidth]{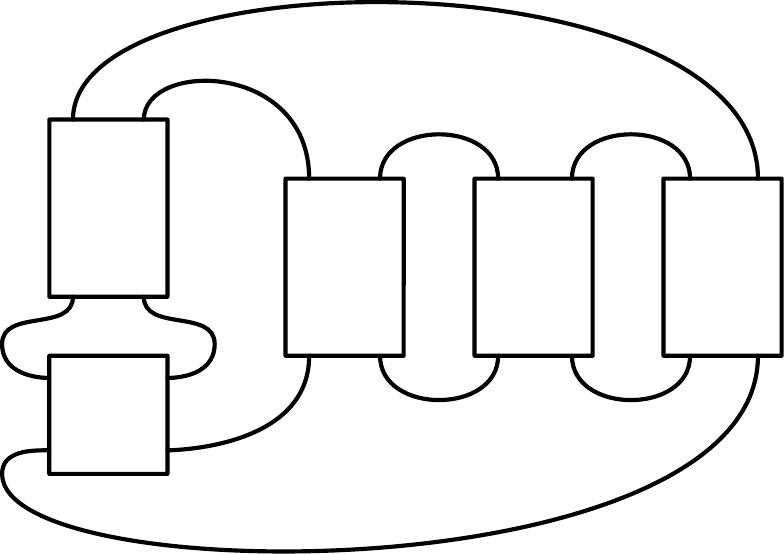}
		\put(11,52){\rotatebox{-90}{$2a+1$}}
		\put(10.5,16){$2b$}
		\put(41,45){\rotatebox{-90}{$2c+1$}}
		\put(65,45){\rotatebox{-90}{$2d+1$}}
		\put(90,45){\rotatebox{-90}{$2e+1$}}
	\end{overpic}
	\caption{(o1)}
	\label{fig:o1}
\end{figure}

\item[(o1')]
Suppose that $K$ is contained in (o1'). 
% By taking a mirror image if necessary, we may assume that $a \ge 1$. 
If $K = M([2a+1,2b], [2c+1],[2d+1],[1])$, then we can see that $\sigma(K) \ne 0$ by the same argument as in the case of (o1). 
Assume that 
\[ K = M([2a+1,2b], [2c+1],[2d+1],[-1]). \] 
Then $K$ is equivalent to $M([2a+1,2b], [2c+1],[-1, 2d])$. 
\begin{enumerate}[leftmargin=17pt]
	\item 
	Suppose that $d \ge 1$. 
	Then $a, c \le -2$ by Lemma~\ref{lem:LT}. %, which contradicts the assumption that $a \ge 1$.  
	\begin{enumerate}[leftmargin=17pt]
		\item Assume that $b \ge 1$. 
		Since $K$ admits an alternating diagram as in Figure~\ref{fig:o1-1}, by Lemma~\ref{lem:signature}, 
		\[ \sigma(K) = (2 |a| -2 +2 |c| -2 +4) -(2|a| -1 + 2|c|-1+1) -1 =0. \]
		We will treat this case later. 
		\item Assume that $b \le -1$. 
		Since $K$ admits an alternating diagram as in Figure~\ref{fig:o1-2}, by Lemma~\ref{lem:signature}, 
		\begin{align*}
			\sigma(K) &= (2 |a| -3 +2 |b| -2 + 2|c| -2 +4) -(2|a| -2 +2|b| -1 + 2|c|-1 + 2) -1= -2. 
		\end{align*}
	\end{enumerate}
	\item Suppose that $d \le -2$. 
	Then we may assume that $c \le -2$ since the case where $c \ge 1$ is the same as (i) by replacing the labels $c$ and $d$. 
    By using a flype, we see that $K=M([2a+1,2b],[-1],[2c+1],[2d+1])$. 
	\begin{enumerate}[leftmargin=17pt]
		\item Assume that $a \ge 1$ and $b \ge 1$. 
		Then we see that $\sigma(K) = 0$ by the same argument, thus we treat these knots later. 
		\item Assume that $a \ge 1$ and $b \le -1$. 
        Since $K$ admits an alternating diagram as in Figure~\ref{fig:o1-3}, by using Lemma~\ref{lem:signature}, we see that $\sigma(K) = -2$. 
	   	\item Assume that $a \le -1$ and $b \le -1$. 
        Since $K$ admits an alternating diagram as in Figure~\ref{fig:o1-4}, by using Lemma~\ref{lem:signature}, we see that $\sigma(K) = -4$.
		In fact, $K$ is a positive knot, and thus, $\sigma(K) < 0$. 
        \item Assume that $a \le -1$ and $b \ge 1$. 
		Then the knot is obtained by two banding from the knot in the case (ii-3). 
        In fact, knots in the case (ii-3) and the case (ii-4) differ only by twists related to the parameter $b$. 
        These two knots can be transformed into the same two-component link as depicted in Figure~\ref{fig:o1-6} by a banding. 
		Since the signature of the knot in the case (ii-3) is $-4$, by Lemma~\ref{lem:signatureBand}, $\sigma(K) \le -2$ holds. 
	\end{enumerate}
\end{enumerate}
Here we treat the alternating Montesinos knots contained in (o1') with $\sigma=0$. 
First we treat the knots of (o1')-(i-1). 
Let 
\[ K = M([2a+1,2b], [2c+1],[-1, 2d]) \text{\ with  } a, c \le -2 \text{ and } \ b,d  \ge 1.\]
By taking the mirror image, $K$ is changed to 
\[ M([2a+1,2b], [2c+1],[1, 2d]) \text{\ with  } a, c \ge 1 \text{ and } \ b,d  \le -1, \]
which is isotopic to $M([2a+1,2b],[1, 2d],[2c+1])$. 
Changing parameter $c$ to $e$, the knot is contained (o2) with $c = 0$. 
Then $K$ admits no purely cosmetic surgeries by Claim~\ref{clm:o2}.

Next we treat the knots of (o1')-(ii-1). 
Let 
\[ K = M([2a+1,2b], [-1], [2c+1],[2d+1]) \text{\ with  } a, b \ge 1 \text{ and } \ c,d  \le -2.\]
Then $K$ is equivalent to the knot depicted in the right side of Figure~\ref{fig:o1-5}, which is obtained by $\overline{t_2'}$~moves on $6_3$. 
Precisely, it is equivalent to the knot in the left side of Figure~\ref{fig:alt} with $a,b \ge 0$, $c,d =0$, and $e,f \ge 1$.

\begin{figure}[htb]
\begin{minipage}[b]{0.48\columnwidth}
    \centering
	\begin{overpic}[width=.6\textwidth]{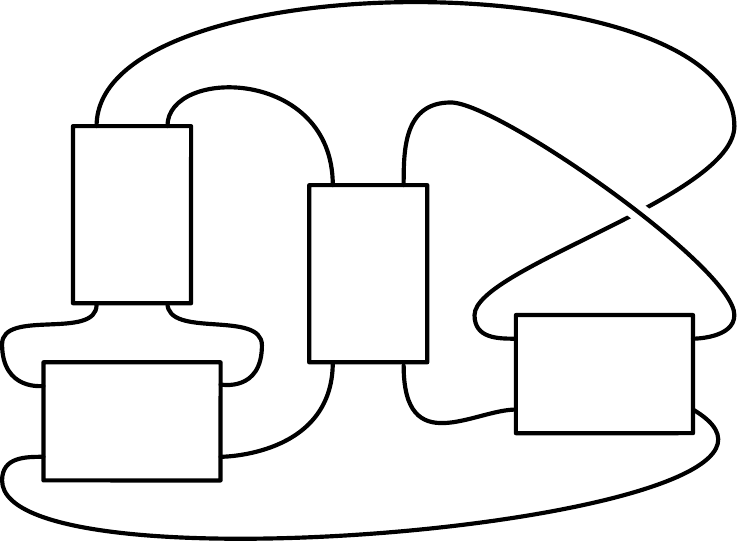}
		\put(15,53){\rotatebox{-90}{$2a+1$}}
		\put(14,13){$2b$}
		\put(46.5,45){\rotatebox{-90}{$2c+1$}}
		\put(78,19.5){$2d$}
	\end{overpic}
	\caption{(o1')-(i-1).}
	\label{fig:o1-1}
  \end{minipage}
  \begin{minipage}[b]{0.48\columnwidth}
    \centering
	\begin{overpic}[width=.6\textwidth]{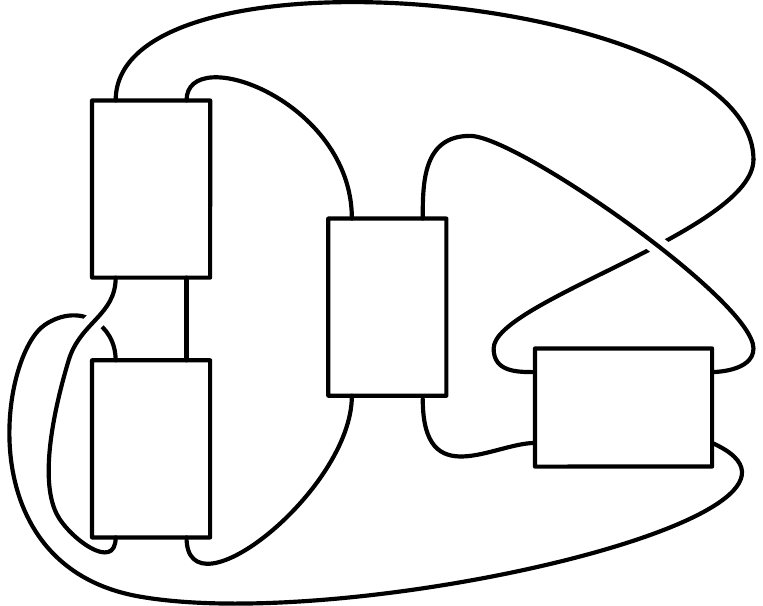}
		\put(17,65){\rotatebox{-90}{$2a+2$}}
		\put(17,30){\rotatebox{-90}{$2b+1$}}
		\put(48,49){\rotatebox{-90}{$2c+1$}}
		\put(78,23){$2d$}
	\end{overpic}
	\caption{(o1')-(i-2).}
	\label{fig:o1-2}
\end{minipage}
\end{figure}
\begin{figure}[htb]
\begin{minipage}[b]{0.48\columnwidth}
    \centering
	\begin{overpic}[width=.6\textwidth]{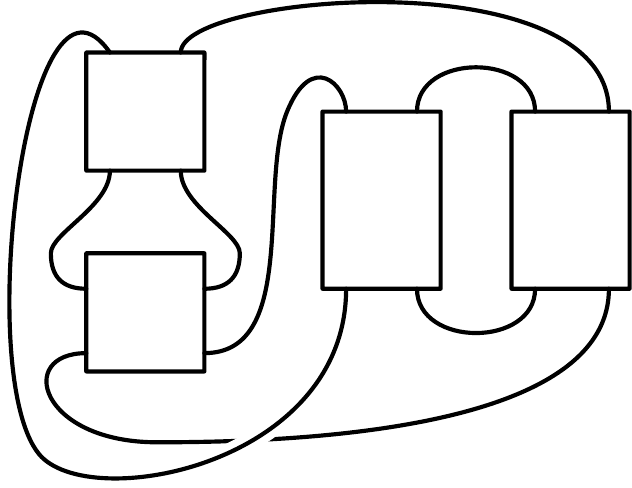}
		\put(20,63){\rotatebox{-90}{$2a$}}
		\put(19,25){$2b$}
		\put(57.5,55.5){\rotatebox{-90}{$2c+1$}}
		\put(88,55.5){\rotatebox{-90}{$2d+1$}}
	\end{overpic}
	\caption{(o1')-(ii-2).}
	\label{fig:o1-3}
  \end{minipage}
  \begin{minipage}[b]{0.48\columnwidth}
    \centering
	\begin{overpic}[width=.55\textwidth]{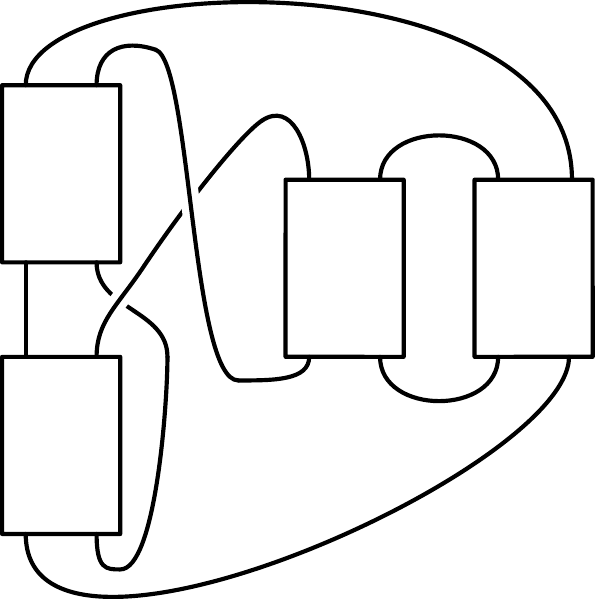}
		\put(7,83){\rotatebox{-90}{$2a+2$}}
		\put(7,37){\rotatebox{-90}{$2b+1$}}
		\put(54,67){\rotatebox{-90}{$2c+1$}}
		\put(85,67){\rotatebox{-90}{$2d+1$}}
	\end{overpic}
	\caption{(o1')-(ii-3).}
	\label{fig:o1-4}
\end{minipage}
\end{figure}
\begin{figure}[htb]
    \centering
	\begin{overpic}[width=\textwidth]{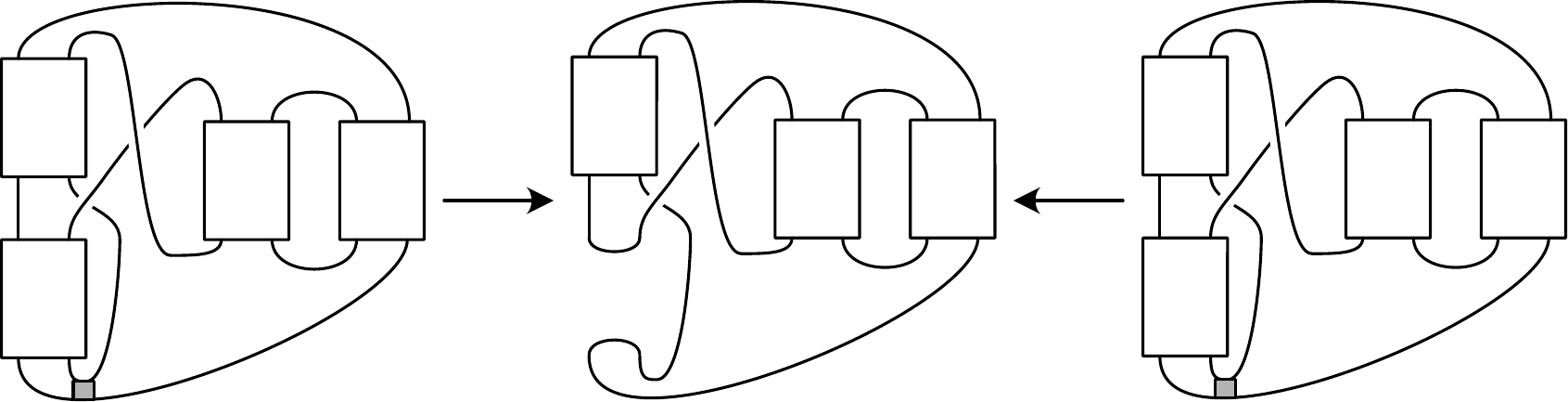}
		\put(2,21.4){\rotatebox{-90}{$2a+2$}}
		\put(2,9.7){\rotatebox{-90}{$2b+1$}}
		\put(14.8,17){\rotatebox{-90}{$2c+1$}}
		\put(23.5,17){\rotatebox{-90}{$2d+1$}}
  		\put(38.2,21.4){\rotatebox{-90}{$2a+2$}}
		\put(51.2,17){\rotatebox{-90}{$2c+1$}}
		\put(59.8,17){\rotatebox{-90}{$2d+1$}}
  		\put(74.6,21.4){\rotatebox{-90}{$2a+2$}}
    	\put(74.6,10.1){\rotatebox{-90}{$-2b-1$}}
		\put(87.8,17){\rotatebox{-90}{$2c+1$}}
		\put(96.4,17){\rotatebox{-90}{$2d+1$}}
  		\put(28,14.3){banding}
   		\put(64.5,14.3){banding}
  	\end{overpic}
	\caption{Knots in (o1')-(ii-4) and (o1')-(ii-3) can be transformed into the same two-component link by a banding.}
	\label{fig:o1-6}
\end{figure}
\begin{figure}[htb]
    \centering
	\begin{overpic}[width=.7\textwidth]{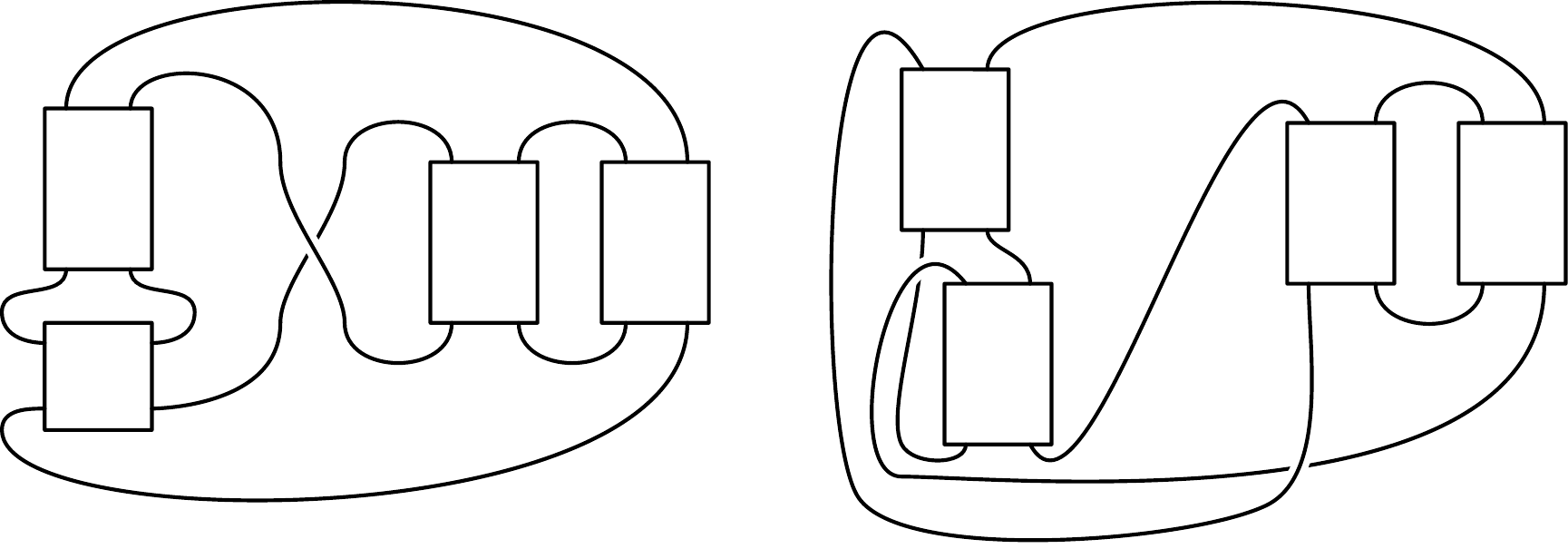}
		\put(5,27){\rotatebox{-90}{$2a+1$}}
		\put(5,10){$2b$}
		\put(30,23.5){\rotatebox{-90}{$2c+1$}}
		\put(41,23.5){\rotatebox{-90}{$2d+1$}}
        \put(48,18){$=$}
  		\put(60,29.5){\rotatebox{-90}{$2a-1$}}
		\put(62.5,15.5){\rotatebox{-90}{$2b-1$}}
		\put(84,26){\rotatebox{-90}{$2c+1$}}
		\put(95.5,26){\rotatebox{-90}{$2d+1$}}
	\end{overpic}
	\caption{(o1')-(ii-1).}
	\label{fig:o1-5}
\end{figure}

\item[(o2)]
Suppose that $K$ is contained in (o2). 
Since $K$ is alternating, $K$ is equivalent to a knot obtained from $7_7$ by $\overline{t_2'}$~moves as discussed in the proof of Theorem~\ref{thm:alt}. 
Thus, $K$ admits no purely cosmetic surgeries.

\item[(o3)]
Suppose that $K$ is contained in (o3). 
By taking a mirror image if necessary, we may assume that $K = M([2a,-3],[2b+1],[2c+1])$. 
Then, by Lemma~\ref{lem:LT}, $b, c \ge 1$. 
\begin{enumerate}[leftmargin=17pt]
	\item Suppose that $a \ge 1$. 
	Then we see that $\sigma(K) = 4$ by using Lemma~\ref{lem:signature}. 
	\item Suppose that $a \le -1$. 
	Then the knot is obtained by two banding from the knot in the case (i). 
	By Lemma~\ref{lem:signatureBand}, $\sigma(K) \ge 2$. 
\end{enumerate}

\item[(o3')]
Suppose that $K$ is contained in (o3'). 
By taking a mirror image if necessary, we may assume that $K = M([2,-3],[2b+1],[2c+1])$. 
Then, by Lemma~\ref{lem:LT}, $b, c \ge 1$. 
Then we see that $\sigma(K) = 4$ by using Lemma~\ref{lem:signature}.

\item[(o4)]
Suppose that $K$ is contained in (o4). 
By taking a mirror image if necessary, we may assume that $a,c,d \ge 1$. 
\begin{enumerate}[leftmargin=17pt]
	\item Suppose that $K= M([2a,-2,2b+1],[2c+1],[2d+1])$. 
	\begin{enumerate}[leftmargin=17pt]
		\item Suppose that $b \ge 1$. 
        Then we see that $\sigma(K) = 4$ by using Lemma~\ref{lem:signature}. 
		\item Suppose that $b \le -2$. 
        Then the knot is obtained by two banding from the knot in the case (i-1). 
	    By Lemma~\ref{lem:signatureBand}, $\sigma(K) \ge 2$. 
	\end{enumerate}
	\item Suppose that $K= M([2a,2,2b+1],[2c+1],[2d+1])$. 
	\begin{enumerate}[leftmargin=17pt]
		\item Suppose that $b \ge 1$. 
		Then $K$ admits an alternating diagram as in Figure~\ref{fig:o4}. 
		By Lemma~\ref{lem:signature}, 
		\[ \sigma(K) = 5-2-1 = 2. \]	
		\item Suppose that $b \le -2$. 
        Then $K$ is equivalent to the knot depicted in the right side of Figure~\ref{fig:o4-1}, which is obtained by $\overline{t_2'}$~moves on $6_3$. 
        Precisely, it is equivalent to the knot in the left side of Figure~\ref{fig:alt} with $a \ge 0$, $b,c,e \ge 1$, and $d = f =0$. 
	\end{enumerate}
\end{enumerate}

\item[(o4')]
Suppose that $K$ is contained in (o4'). 
By taking a mirror image if necessary, we may assume that $K= M([2,-2,2b+1],[2c+1],[2d+1])$ with $c,d \ge 1$. 
\begin{enumerate}[leftmargin=17pt]
	\item Suppose that $b \ge 1$. 
    Then we see that $\sigma(K) = 4$ by using Lemma~\ref{lem:signature}. 
	\item Suppose that $b \le -2$. 
    Then the knot is obtained by two banding from the knot in the case (i). 
	By Lemma~\ref{lem:signatureBand}, $\sigma(K) \ge 2$. 
\end{enumerate}

\begin{figure}[htb]
\begin{minipage}[b]{0.48\columnwidth}
    \centering
	\begin{overpic}[width=.6\textwidth]{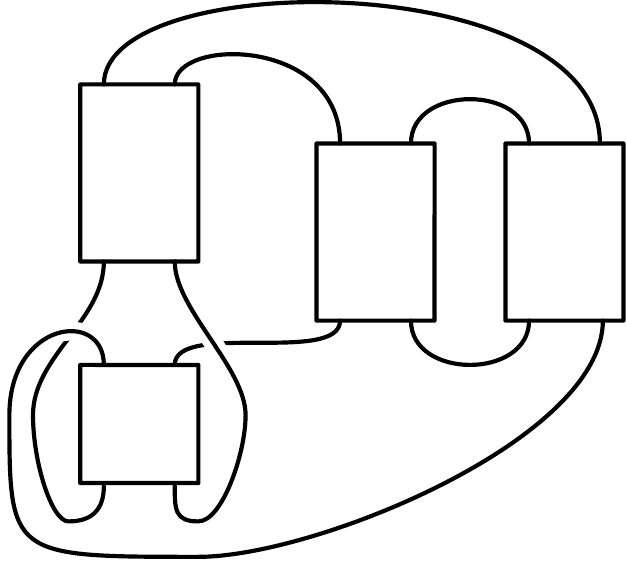}
		\put(19,73){\rotatebox{-90}{$2a-1$}}
		\put(19,26){\rotatebox{-90}{$2b$}}
		\put(57.5,64){\rotatebox{-90}{$2c+1$}}
		\put(88,64){\rotatebox{-90}{$2d+1$}}
	\end{overpic}
	\caption{(o4) with $b\ge 1$.}
	\label{fig:o4}
  \end{minipage}
  \begin{minipage}[b]{0.48\columnwidth}
    \centering
	\begin{overpic}[width=.5\textwidth]{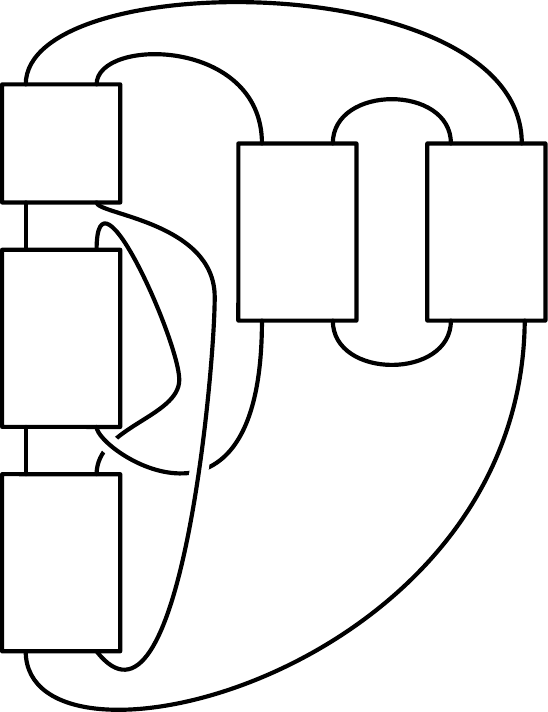}
		\put(6,83){\rotatebox{-90}{$2a$}}
		\put(6,30){\rotatebox{-90}{$2b-2$}}
		\put(6,61){\rotatebox{-90}{$2c-1$}}
		\put(39,76){\rotatebox{-90}{$2d+1$}}
        \put(66,76){\rotatebox{-90}{$2e+1$}}
	\end{overpic}
	\caption{(o5)-(i).}
	\label{fig:o5}
\end{minipage}
\end{figure}

\begin{figure}[htb]
    \centering
	\begin{overpic}[width=.7\textwidth]{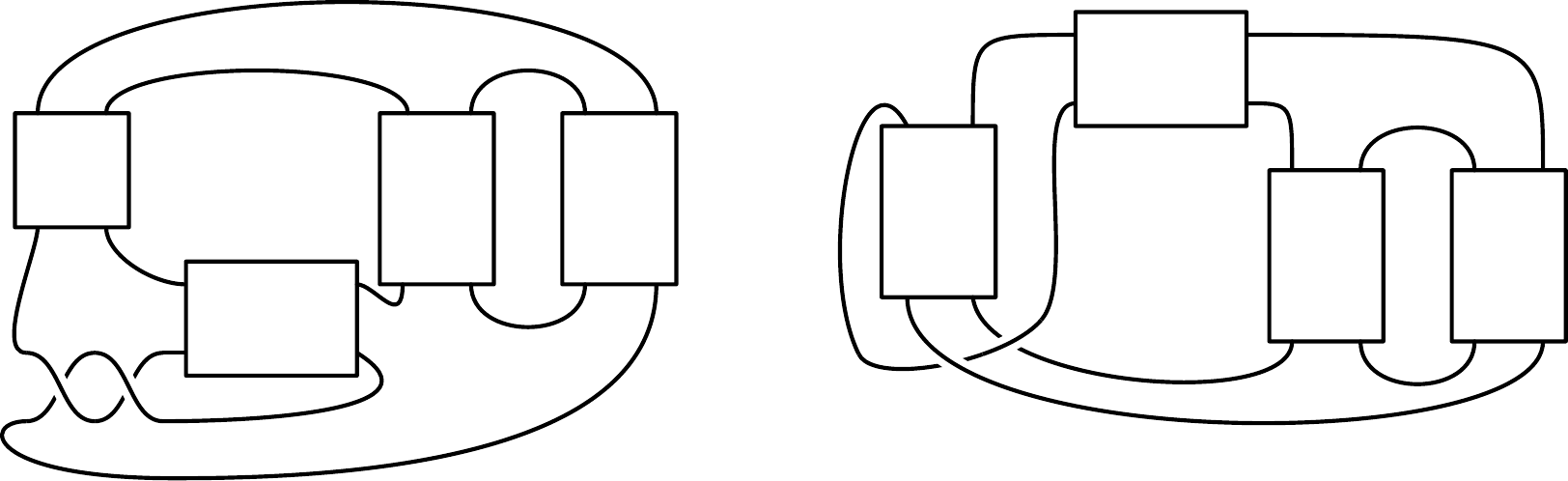}
		\put(4,21.5){\rotatebox{-90}{$2a$}}
		\put(13.5,9.5){$2b+1$}
		\put(26.5,22){\rotatebox{-90}{$2c+1$}}
		\put(38.5,22){\rotatebox{-90}{$2d+1$}}
        \put(48,15){$=$}
  		\put(69.5,25.5){$2a-1$}
		\put(58.6,21.5){\rotatebox{-90}{$2b+1$}}
		\put(83.5,18.5){\rotatebox{-90}{$2c+1$}}
		\put(95,18.5){\rotatebox{-90}{$2d+1$}}
	\end{overpic}
	\caption{(o4)-(ii-2).}
	\label{fig:o4-1}
\end{figure}

\item[(o5)]
Suppose that $K$ is contained in (o5). 
By taking a mirror image if necessary, we may assume that $a,d,e \ge 1$. 
\begin{enumerate}[leftmargin=17pt]
	\item Suppose that $b \ge 1$ and $c \ge 1$. 
	Then $K$ admits an alternating diagram as in Figure~\ref{fig:o5}. 
	By Lemma~\ref{lem:signature}, 
	\[ \sigma(K) = 5-0-1 = 4. \]	
	\item Suppose that $b \ge 1$ and $c \le -1$. 
    Then the knot is obtained by two banding from the knot in the case (i). 
	By Lemma~\ref{lem:signatureBand}, $\sigma(K) \ge 2$. 
	\item Suppose that $b \le -1$ and $c \ge 1$. 
    Then the knot is obtained by two banding from the knot in the case (i). 
	By Lemma~\ref{lem:signatureBand}, $\sigma(K) \ge 2$. 
	\item Suppose that $b \le -1$ and $c \le -1$. 
	Then $K$ is obtained from $6_3$ by $\overline{t_2'}$~moves. 
    Precisely, it is equivalent to the knot in the left side of Figure~\ref{fig:alt} with $a,b,c,d,f \ge 1$, and $e =0$. 
\end{enumerate}

\item[(e1)] 
Suppose that $K$ is contained in (e1). 
By taking a mirror image if necessary, we may assume that $a \ge 1$. 
Then $b,d \ge 1$ by Lemma~\ref{lem:LT}. 
\begin{enumerate}[leftmargin=17pt]
	\item Suppose that $c \ge 1$ and $e \ge 1$. 
	Then $K$ admits an alternating (negative) diagram, and by Lemma~\ref{lem:signature}, we see that 
	\[ \sigma(K) = 4. \] 
	\item Suppose that $b \ge 1$ and $e \le -1$. 
	Then the knot is obtained by two banding from the knot of the case (i). 
	By Lemma~\ref{lem:signatureBand}, $\sigma(K) \ge 2$. 
	\item Suppose that $b \le -1$ and $e \ge 1$. 
	The knot in this case is also obtained by two banding from the knot of the case (i). 
	By Lemma~\ref{lem:signatureBand}, $\sigma(K) \ge 2$. 
	\item Suppose that $b \le -1$ and $e \le -1$. 
	Then $K$ is obtained from $7_7$ by $\overline{t_2'}$~moves. 
	Thus, by Theorem~\ref{thm:alt}, $K$ admits no purely cosmetic surgeries. \qedhere
\end{enumerate}
\end{description}
\end{proof}

\begin{proof}[Proof of Corollary~\ref{cor:altMontesinos}]
An alternating Montesinos knot $K$ stated in Theorem~\ref{thm:altMontesinos} is contained in each of (o1'), (o4), or (o5). 
If $K$ is contained in (o1'), then it is equivalent to $M([2a+1,2b], [2c+1],[-1, 2d])$ which is a Montesinos knot of length three. 
If $K$ is contained in (o4), or (o5), then $K$ is also a Montesinos knot of length three. 
This completes the proof of Corollary~\ref{cor:altMontesinos}. 
\end{proof}

\subsection*{Acknowledgements}
The authors would like to thank Fusashi Nakamura for suggestions on the proof of Claim~\ref{clm:o2}. %for suggestions about Claims~\ref{clm:a2w3} and \ref{clm:o2}. 
They also thank the referee for his/her careful reading and useful suggestions. 
The first author is partially supported by JSPS KAKENHI Grant Number 22K03301. 
% 18K03287. 
The second author is partially supported by JSPS KAKENHI Grant Number 22K03324. 
% 19K03483.

\bibliographystyle{amsplain}
\bibliography{IJ2023}

\appendix

\setcounter{theorem}{0}
\renewcommand\thetheorem{A.\arabic{theorem}}
\addcontentsline{toc}{section}{Appendix}

\section*{Appendix}

Here, for further studies, we give calculations of $a_2$ and $w_3$ for Montesinos knots of genus two in the families (o1), (o3), (o4), (o5), (e1) listed in Proposition~\ref{prop:Monteg2}. 
First we give formulae for pretzel knots without a proof since one can prove them by elementary calculations. 
See \cite[Proposition 6.5]{IchiharaItoSaito} or \cite[Lemma 2.3]{Varvarezos22} for example.

\begin{lemma}\label{lem:pretzel}
	Let $P(2x+1,2y+1,2z+1)$ be a three-strands pretzel knot of odd type, where $x,y,z$ are integers. 
	Then 
	\begin{align*}
		a_2 ( P(2x+1,2y+1,2z+1) ) 
		&= (x+1)(y+1)(z+1) - xyz, \\ 
		w_3 ( P(2x+1,2y+1,2z+1) ) 
		&= ( xy(x+y) + yz (y+z) + zx (z+x) + 4 xyz \\ 
        & \quad + (x+y+z)^2 + 2 (xy + yz + zx) + 3(x+y+z) + 2 )/4. 
	\end{align*}
\end{lemma}

% \begin{proof}
% 	Note that $P(1,1,1)$ is the left-handed trefoil knot and $a_2 ( P(1,1,1) ) = 1$, $w_3 ( P(1,1,1) ) = \dfrac{1}{2}$. 
% 	First we calculate $a_2$. 
% 	Applying the skein relation~\eqref{eq:a2skein} repeatedly, 
% 	\begin{align*}
% 		a_2 ( P(2x+1,2y+1,2z+1) ) 
% 		&= a_2 ( P(1,2y+1,2z+1) ) + x(y+z+1) \\ 
% 		&= a_2 ( P(1,1,2z+1) ) + y(z+1) + x(y+z+1) \\ 
% 		&= a_2 ( P(1,1,1) ) + z + y(z+1) + x(y+z+1) \\ 
% 		% &= 1+ z + y(z+1) + x(y+z+1) \\ 
% 		&= (x+1)(y+1)(z+1) - xyz .
% 	\end{align*}
% 	Next we calculate $w_3$. 
% 	Applying Lemma~\ref{lem:w3twists} repeatedly, 
% 	\begin{align*}
% 		w_3 ( P(2x+1,2y+1,2z+1) ) 
% 		&= w_3 ( P(1,2y+1,2z+1) ) + \dfrac{x}{2}a_2(P(1,x,y)) \\ 
% 		& \quad + \dfrac{x}{4} (-y-z-1)(-y-z-1-x) \\ 
% 		&= \( w_3 ( P(1,1,2z+1) ) + \dfrac{y}{2}(1+z) + \dfrac{y}{4}(z+1)(y+z+1) \) \\  
% 		& \quad + \dfrac{x}{2}(1+y+z+yz) + \dfrac{x}{4} (y+z+1)(x+y+z+1) \\ 
% 		&= \( \dfrac12 + \dfrac{z}{2} + \dfrac{z}{4}(-1)(-1-z) \) + \dfrac{y}{2}(1+z) + \dfrac{y}{4}(z+1)(y+z+1) \\ 
% 		& \quad + \dfrac{x}{2}(1+y+z+yz) + \dfrac{x}{4} (y+z+1)(x+y+z+1) \\ 
% 		&= ( xy(x+y) + yz (y+z) + zx (z+x) + 4 xyz \\ 
% 		& \quad + (x+y+z)^2 + 2 (xy + yz + zx) + 3(x+y+z) + 2 )/4. \qedhere 
% 	\end{align*}
% \end{proof}

Let $a, b, c, d, e$ be non-zero integers. 

\begin{claim}\label{clm:o3}
%	For $K = K(a,b,c) = M\left([2a+1, 3], [2b+1], [2c+1] \right)$, 
	For $K = K(a,b,c) = M\left([2a, 3], [2b+1], [2c+1] \right)$, 
    \begin{align*}
		a_2(K) &= ab + bc + ca +a -b-c, \\ %\label{eq:o3a2}\\
		\begin{split} %\label{eq:o3w3}	
			w_3(K) &= \dfrac{c}{4}(1-c) -\dfrac{bc}{2} + \dfrac{b}{4}(c-1)(b+c-1) - \dfrac{a}{2} \\
			& \quad + \dfrac{1}{4}\( (a+1)(bc - b-c) + a(b+c+1)(a+b+c+1) \). 
		\end{split}
	\end{align*}
\end{claim}
\begin{proof}
	% Note that $K(0,0,0,+)$ and $K(0,0,0,-)$ are equivalent to the torus knot of type $(2,-5)$ and the unknot respectively. 
	Note that $K(0,0,0)$ is equivalent to the unknot. 
	First we calculate $a_2$. 
	Applying the skein relation~\eqref{eq:a2skein} repeatedly, 
	\begin{align*}
		a_2\(K(a,b,c)\) 
		&= a_2\(K(0,b,c)\) + a(b+c+1) \\ 
		&= a_2\(K(0,0,c)\) + b(c-1) + a(b+c+1) \\ 
		&= -c + b(c-1) + a(b+c+1) \\ 
		&= ab + bc + ca +a -b-c. 
	\end{align*} 
	Next we calculate $w_3$. 
	In this case, we cannot use Lemma~\ref{eq:TwistFormula} directly since one of the components of $K' \cup K''$ is non-trivial. 
	Using Equation~\eqref{eq:w3skein} repeatedly, we have  
	\begin{align*}
		w_3(K(a,b,c)) 
		&= w_3(K(0,b,c)) - \dfrac{a}{2} \\ 
		&\quad + \dfrac{1}{4} \( \sum_{k=0}^a a_2\(K(k,b,c)\) + \sum_{k=1}^{a-1} a_2\(K(k,b,c)\) + a (b+c+1)^2 \). 	
	\end{align*}
	Since $a_2\(K(a,b,c)\) = a(b+c+1) + bc -b -c$, we have 
	\begin{align*}
		\sum_{k=0}^a a_2(K(k,b,c)) + \sum_{k=1}^{a-1} a_2(K(k,b,c)) 
		&= a_2(K(0,b,c)) + a_2(K(a,b,c)) \\ 
		& \quad + 2 \cdot \dfrac{1}{2}(a-1)a(b+c+1) + (a-1)(bc - b-c)\\ 
		&=  bc -b-c + ab + bc + ca +a-b-c \\ 
		& \quad + (a-1)a(b+c+1)+(a-1)(bc-b-c) \\ 
		&= (a+1)(bc - b-c) + a^2(b+c+1) . 
	\end{align*}
	Thus, we have 
	\begin{align*}
		w_3(K(a,b,c)) 
		&= w_3(K(0,b,c)) - \dfrac{a}{2} \\ 
		& \quad + \dfrac{1}{4}\( (a+1)(bc - b-c) + a(b+c+1)(a+b+c+1) \) \\ 
		&= w_3(K(0,0,c)) + \dfrac{b}{2}a_2(K(0,0,c)) + \dfrac{b}{4}(c-1)(1-b-c) - \dfrac{a}{2} \\ 
		& \quad + \dfrac{1}{4}\( (a+1)(bc - b-c) + a(b+c+1)(a+b+c+1) \) \\ 
		&= w_3(K(0,0,0)) + \dfrac{c}{2}a_2(K(0,0,0)) + \dfrac{c}{4}(1-c)
		-\dfrac{bc}{2} + \dfrac{b}{4}(c-1)(b+c-1) \\
		& \quad  - \dfrac{a}{2} + \dfrac{1}{4}\( (a+1)(bc - b-c) + a(b+c+1)(a+b+c+1) \) \\ 
		&= \dfrac{c}{4}(1-c) -\dfrac{bc}{2} + \dfrac{b}{4}(c-1)(b+c-1) - \dfrac{a}{2} \\
		& \quad + \dfrac{1}{4}\( (a+1)(bc - b-c) + a(b+c+1)(a+b+c+1) \) . \qedhere 
	\end{align*}
\end{proof}

\begin{claim}\label{clm:o1}
	For $K = K(a,b,c,d,e) = M\left([2a+1, 2b], [2c+1], [2d+1], [2e+1] \right)$, 
	\begin{align*}
		a_2(K) &= (c+1)(d+1)(e+1)-cde + b(a+c+d+e+2), \\
		w_3(K) &= ( cd(c+d) + de (d+e) + ec (e+c) + 4 cde \\ 
		& \quad + (c+d+e)^2 + 2 (cd + de + ec) + 3(c+d+e) + 2 )/4 \\ 
		& \quad + \dfrac{b}{2} \( (c+1)(d+1)(e+1) -cde \) \\ 
		& \quad + \dfrac{b}{4}(a+c+d+e+2)(a+b+c+d+e+2) \, . 
	\end{align*}
\end{claim}
\begin{proof}
	Note that $K(a,0,c,d,e)$ is equivalent to the pretzel knot $P(2c+1,2d+1,2e+1)$. 
	Thus, we apply Lemma~\ref{lem:pretzel} when the pretzel knot appears. 
	First we calculate $a_2$. 
	Applying the skein relation~\eqref{eq:a2skein} to the twist box labeled $2b$, we have
	\begin{align*}
		a_2\(K(a,b,c,d,e)\) 
		&= a_2\(K(a,0,c,d,e)\) + b(a+c+d+e+2) \\ 
		&= (c+1)(d+1)(e+1)-cde + b(a+c+d+e+2). 
	\end{align*} 
	Next we calculate $w_3$. 
	Applying Lemma~\ref{lem:w3twists} repeatedly, 
	\begin{align*}
		w_3(K(a,b,c,d,e)) 
		&= w_3(K(a,0,c,d,e)) + \dfrac{b}{2} a_2(K(a,0,c,d,e)) \\ 
		& \quad + \dfrac{b}{4}(a+c+d+e+2)(a+b+c+d+e+2) \\ 
		&= ( cd(c+d) + de (d+e) + ec (e+c) + 4 cde \\ 
		& \quad + (c+d+e)^2 + 2 (cd + de + ec) + 3(c+d+e) + 2 )/4 \\ 
		& \quad + \dfrac{b}{2} \( (c+1)(d+1)(e+1) -cde \) \\ 
		& \quad + \dfrac{b}{4}(a+c+d+e+2)(a+b+c+d+e+2) \, . \qedhere 
	\end{align*}
\end{proof}

\begin{claim}\label{clm:o4}
	For $K = K(a,b,c,d) = M\left([2a, 2, 2b+1], [2c+1], [2d+1] \right)$, 
	\begin{align*}
		a_2(K) &= a + ac + ad + cd + bc + bd, \\
		w_3(K) &= (cd(c+d) + a^2(1+c+d) +a(c^2+1+2d+d^2+2c+4cd) \\ 
		& \quad + 2b(a+ac+ad+cd) + b (c+d) (b+c+d) )/4
	\end{align*}
\end{claim}
\begin{proof}
	Note that $K(a,0,b,c) = M\left([2a, 2, 1], [2c+1], [2d+1] \right)$ is equivalent to the pretzel knot $P(2a-1,2c+1,2d+1)$. 
	First we calculate $a_2$. 
	Applying the skein relation~\eqref{eq:a2skein} to the twist box labeled $2b+1$ repeatedly, 
	\begin{align*}
		a_2\(K(a,b,c,d)\) 
		&= a_2\( P(2a-1,2c+1,2d+1) \) + b(c+d) \\ 
		&= 1+ (a-1) + c + d + (a-1)c + (a-1)d + cd + b(c+d) \\ 
		&= a + ac + ad + cd + bc + bd. 
	\end{align*} 
	Next we calculate $w_3$. 
	Applying Lemma~\ref{lem:w3twists} repeatedly, 
	\begin{align*}
		w_3(K(a,b,c,d)) 
		&= w_3(P(2a-1,2c+1,2d+1)) + \dfrac{b}{2}(a + ac + ad + cd) + \dfrac{b}{4}(c+d)(b+c+d) \\ 
		&= (cd(c+d) + a^2(1+c+d) +a(c^2+1+2d+d^2+2c+4cd) \\ 
		& \quad + 2b(a+ac+ad+cd) + b (c+d) (b+c+d) )/4 . \qedhere 
	\end{align*}
\end{proof}

\begin{claim}\label{clm:o5}
	For $K = K(a,b,c,d,e) = M\left([2a+1, 2b, 2c], [2d+1], [2e+1] \right)$, 
	\begin{align*}
		a_2(K) &= (a+1)(d+1)(e+1)-ade + c(b+d+e+1), \\
		w_3(K) &= (a^2(1+d+e) +(1+d)(1+e)(2+d+e) + a(3+4d + d^2 + 4e +e^2 +4de) \\ 
		& \quad + 2c((a+1)(d+1)(e+1)-ade) + c(b+d+e+1)(b+c+d+e+1) )/4
	\end{align*}
\end{claim}
\begin{proof}
	Note that $K(a,b,0,d,e) = M\left([2a+1, 2b, 0], [2d+1], [2e+1] \right)$ is equivalent to the pretzel knot $P(2a+1,2d+1,2e+1)$. 
	First we calculate $a_2$. 
	Applying the skein relation~\eqref{eq:a2skein} to the twist box labeled $2b$ repeatedly, 
	\begin{align*}
		a_2\(K(a,b,c,d,e)\) 
		&= a_2\( P(2a+1,2d+1,2e+1) \) + c(b+d+e+1) \\ 
		&= (a+1)(d+1)(e+1)-ade + c(b+d+e+1). 
	\end{align*} 
	Next we calculate $w_3$. 
	Applying Lemma~\ref{lem:w3twists} repeatedly, 
	\begin{align*}
		w_3(K(a,b,c,d,e)) 
		&= w_3(P(2a+1,2d+1,2e+1)) + \dfrac{c}{2}((a+1)(d+1)(e+1)-ade) \\ 
		& \quad + \dfrac{c}{4}(b+d+e+1)(b+c+d+e+1) \\ 
		&= (a^2(1+d+e) +(1+d)(1+e)(2+d+e) + a(3+4d + d^2 + 4e +e^2 +4de) \\ 
		& \quad + 2c((a+1)(d+1)(e+1)-ade) + c(b+d+e+1)(b+c+d+e+1) )/4 . \qedhere 
	\end{align*}
\end{proof}

\begin{claim}\label{clm:e1}
	For $K = K(a,b,c,d,e) = M\left([2a], [2b,2c] ,[2d,2e] \right)$, 
	\begin{align*}
		a_2(K) &= bc + de + ac + ae, \\
		w_3(K) &= (bc(b+c) + de(d+e) + 2a(bc + de) + a(c+e)(a+c+e) )/4
	\end{align*}
\end{claim}
\begin{proof}
	Note that $K(0,b,c,d,e)$ is equivalent to $\mathrm{DT}(2b,2c) \sharp \mathrm{DT}(2d,2e)$. 
	First we calculate $a_2$. 
	Applying the skein relation~\eqref{eq:a2skein} to the twist box labeled $2a$ repeatedly, 
	\begin{align*}
		a_2\(K(a,b,c,d,e)\) 
		&= a_2\( K(0,b,c,d,e) \) +a (c+e) \\ 
		&= bc + de + ac + ae.  
	\end{align*}
	Next we calculate $w_3$. 
	Applying Lemma~\ref{lem:w3twists} repeatedly, 
	\begin{align*}
		w_3(K(a,b,c,d,e)) 
		&= w_3(K(0,b,c,d,e)) + \dfrac{a}{2}(bc + de) + \dfrac{a}{4}(c+e)(a+c+e) \\ 
		&= w_3(\mathrm{DT}(2b,2c)) + w_3(\mathrm{DT}(2d,2e)) + \dfrac{a}{2}(bc + de) + \dfrac{a}{4}(c+e)(a+c+e) \\ 
		&= (bc(b+c) + de(d+e) + 2a(bc + de) + a(c+e)(a+c+e) )/4 . \qedhere 
	\end{align*} 
\end{proof}

\end{document}